\documentclass[reqno,11pt]{amsart}
\usepackage{hyperref}
\usepackage{amssymb}
\usepackage{amsmath}
\usepackage{amsthm}
\usepackage[latin1]{inputenc}
\usepackage[left=1.5cm, right=1.5cm, top=1.5cm, bottom=1.5cm]{geometry}

\newtheorem{Theorem}{Theorem } [section]
\newtheorem{lemma}[Theorem]{Lemma}
\newtheorem{corollary}[Theorem]{Corollary}
\newtheorem{proposition}[Theorem]{Proposition}

\newtheorem{Definition}{Definition}

\numberwithin{equation}{section}

\DeclareMathOperator{\sgn}{sgn}
\DeclareMathOperator{\supp}{supp}
\DeclareMathOperator{\re}{Re}
\DeclareMathOperator{\meas}{meas}
\begin{document}

\title{The Cauchy problem for a family of two-dimensional fractional\\
Benjamin-Ono equations }
\author{Eddye Bustamante, Jos\'e Jim\'enez Urrea and Jorge Mej\'{\i}a}

\address{Eddye Bustamante M., Jos\'e Jim\'enez Urrea, Jorge Mej\'{\i}a L. \newline
Departamento de Matem\'aticas\\Universidad Nacional de Colombia\newline
A. A. 3840 Medell\'{\i}n, Colombia}
\email{eabusta0@unal.edu.co, jmjimene@unal.edu.co, jemejia@unal.edu.co}

\begin{abstract}
In this work we prove that the initial value problem (IVP) associated to the fractional two-dimensional Benjamin-Ono equation
$$\left. \begin{array}{rl} u_t+D_x^{\alpha} u_x +\mathcal Hu_{yy} +uu_x &\hspace{-2mm}=0,\qquad\qquad (x,y)\in\mathbb R^2,\; t\in\mathbb R,\\ u(x,y,0)&\hspace{-2mm}=u_0(x,y), \end{array} \right\}\,,$$
where $0<\alpha\leq1$, $D_x^{\alpha}$ denotes the operator defined through  the Fourier transform by
\begin{align}
(D_x^{\alpha}f)\widehat{\;}(\xi,\eta):=|\xi|^{\alpha}\widehat{f}(\xi,\eta)\,,
\end{align}
 and $\mathcal H$ denotes the Hilbert transform with respect to the variable $x$, is locally well posed in the Sobolev space $H^s(\mathbb R^2)$ with $s>\dfrac32+\dfrac14(1-\alpha)$.
\end{abstract}

\maketitle
\section{Introduction}

In this article we consider a class of initial value problems (IVP) associated to the family of fractional two-dimensional Benjamin-Ono (BO) equations
\begin{align}
\left. \begin{array}{rl}
u_t+D_x^{\alpha} u_x +\mathcal Hu_{yy}+uu_x &\hspace{-2mm}=0,\qquad\qquad (x,y)\in\mathbb R^2,\; t\geq 0,\;u(x,y,t)\in\mathbb R,\\
u(x,y,0)&\hspace{-2mm}=u_0(x,y),
\end{array} \right\}\label{BO}
\end{align}
where $0<\alpha\leq1$, $D_x^{\alpha}$ denotes the operator defined through  the Fourier transform by
\begin{align}
(D_x^{\alpha}f)\widehat{\;}(\xi,\eta):=|\xi|^{\alpha}\widehat{f}(\xi,\eta)\,,
\end{align}
 and $\mathcal H$ denotes the Hilbert transform with respect to the variable $x$, which is defined through the Fourier transform by
\begin{align} 
(\mathcal H f)\widehat{\;\;}(\xi,\eta):= -i \sgn(\xi)\widehat f(\xi,\eta)\;.
\end{align} 

When $\alpha=1$ the equation in \eqref{BO}, called Shrira equation, is a bidimensional generalization of the BO equation 
\begin{align}
u_t+\mathcal H u_{xx} +u \partial_x u =0,\qquad\qquad x\in\mathbb R,\; t\geq 0,\label{BO1}
\end{align}
and was deduced by Pelinovsky and Shrira in \cite{PS1995}  in connection with the propagation of long-wave weakly nonlinear two-dimensional perturbations in parallel boundary-layer type shear flows. Very recently, Esfahany and Pastor in \cite{EP2017} studied for this equation existence, regularity and decay properties of solitary waves. Different two-dimensional generalizations of the BO equation can be found, among others, in the references \cite{AM}, \cite{K2006}, \cite{AC1991}, and \cite{AS1980}.\\

Using the abstract theory developed by Kato in \cite{Ka1975} and \cite{Ka1979}, it can be established the local well-posedness (LWP) of the IVP \eqref{BO} in $H^s(\mathbb R^2)$ with $s>2$ (see Lemma \ref{L5.2} below for $\alpha\in(0,1]$, and \cite{PS2015} for $\alpha=1$). However, this approach does not employ sufficiently the dispersive terms in the equation in order to have LWP in Sobolev spaces  with less regularity.\\ 

Ponce in \cite{P1991} and Koch and Tzvetkov in \cite{KT2003} used the dispersive character of the linear part  of the unidimensional BO equation \eqref{BO1} to prove the LWP of the IVP associated to this equation in Sobolev spaces $H^s(\mathbb R)$, with $s=3/2$ and $s>5/4$, respectively. In this manner, they improved the classical result ($s>3/2$) obtained from a nonlinear commutator argument (see \cite{KaPo1988} and  \cite{S1979}).\\

In particular, in \cite{KT2003} the improvement is a consequence of a nonlinear estimate, which, by means of a standard Littlewood-Paley decomposition, follows from a Strichartz type inequality for a linearized version of the BO equation.\\

In \cite{CP2016} Cunha and Pastor applied the technique, introduced by Koch and Tzvetkov in \cite{KT2003}, in order to study the Benjamin-Ono-Zakharov-Kuznetsov equation in low regularity Sobolev spaces ($H^s(\mathbb R^2)$ with $s>11/8$).\\

In \cite{KeKoe2003} Kenig and Koenig, studying the unidimensional BO equation and based on previous ideas of Koch and Tzvetkov, obtained a refined version of the Strichartz estimate, by dividing the time interval into small subintervals, whose length depends on the spatial frequency of the function. This new Strichartz estimated allowed them to establish LWP of the Cauchy problem in $H^s(\mathbb R)$, with $s> 9/8$.\\

Kenig in \cite{Ke2004}, in the context of the Kadomtsev-Petviashvili equation (KP-I), adapted the argument of Koch and Tzvetkov in \cite{KT2003}, to obtain a refined Strichartz estimate (see Lemma 1.7 in \cite{Ke2004}), similar to that in \cite{KeKoe2003}, in order to prove the LWP of the IVP of the KP-I equation in
\begin{align}
Y_s:=\{\phi\in L^2(\mathbb R^2)\mid \|\phi\|_{L^2}+\|(1-\partial_x^2)^{s/2}\phi\|_{L^2}+\|\partial_x^{-1}\partial_y\phi\|_{L^2}<\infty\}\,,
\end{align}
for $s>3/2$.\\

As it is observed by Linares, Pilod and Saut in \cite{LPS2014}, the study of the well-posedness for nonlinear dispersive equations is based on the comparison between nonlinearity and dispersion. It is usual to fix the dispersion and vary the nonlinearity. The approach we use in this work, inspired in \cite{LPS2014}, is to fix the quadratic nonlinearity and vary the dispersion.\\

In \cite{LPS2014} Linares, Pilod and Saut investigated how a weakly dispersive perturbation  of the inviscid Burgers equation enlarges the space of resolution of the local Cauchy problem. Specifically they studied the family of fractional KdV equations 
\begin{align}
u_t-D_x^{\alpha} u_x +u \partial_x u =0,\qquad\qquad x\in\mathbb R,\; t\geq 0,\;u(x,t)\in\mathbb R, \,\label{fKdV}
\end{align}
in the less dispersive case ($\alpha\in(0,1)$), and for $s>3/2-3\alpha/8$ they proved the LWP of the Cauchy problem for the equation \eqref{fKdV} in the Sobolev space $H^s(\mathbb R)$.\\

Using the same approach as that used in \cite{LPS2014}, Linares, Pilod and Saut in \cite{LPS2017} considered the family of fractional Kadomtsev-Petviashvili equations
\begin{align}
u_t+uu_x-D_x^{\alpha}u_x\mp\partial_x^{-1}u_{yy}=0\,,\quad0<\alpha\leq2\,,\label{KPF}
\end{align}
and established the LWP in $Y_s$ of the IVP associated to the equations \eqref{KPF}, with $s>2-\alpha/4$.\\

The main ingredient in the well-posedness analysis  in the works \cite{LPS2014} and \cite{LPS2017} is a Strichartz estimate, similar to that obtained by Kenig and Koenig in \cite{KeKoe2003} and by Kenig in \cite{Ke2004}.\\
 
Inspired by the works \cite{LPS2014} and \cite{LPS2017}, in this paper we also study the relation between the amount of dispersion and the size of the Sobolev space in order to have LWP of the family of IVPs \eqref{BO}.\\

Following the scheme developed in \cite{LPS2017}, we also obtain a Strichartz estimate (see Lemma \ref{L4.11} below), which, combined with some energy estimates, allows us to prove the LWP of the IVP \eqref{BO} in $H^s(\mathbb R^2)$, with $s>3/2+(1-\alpha)/4$. This is the main result of this work and its statement is as follows.

\begin{Theorem}\label{T4.1} Let $\alpha\in(0,1]$. Let us define $s_{\alpha}:=\frac32+\frac14(1-\alpha)=\frac74-\frac\alpha4$ and assume that $s>s_{\alpha}$. Then, for any $u_0\in H^s(\mathbb R^2)$, there exist a positive time $T=T(\|u_0\|_{H^s})$ and a unique solution of the IVP \eqref{BO} in the class
$$C([0,T];H^s(\mathbb R^2))\cap L^1([0,T];W_{xy}^{1,\infty}(\mathbb R^2)).$$
Moreover, for any $0<T'<T$, there exists a neighborhood $\mathcal U$ of $u_0$ in $H^s(\mathbb R^2)$ such that the flow map datum-solution
\begin{align*}
S^s_{T'}:\mathcal U&\to C([0,T'];H^s(\mathbb R^2))\\
v_0&\mapsto v,
\end{align*}
is continuous.
\end{Theorem}
In the previous theorem $W_{x,y}^{1,\infty}(\mathbb R^2)$ denotes the Sobolev space given by the closure of the Schwartz functions under the norm
\[\|f\|_{W_{x,y}^{1,\infty}(\mathbb R^2)}:=\|f\|_{L_{xy}^{\infty}(\mathbb R^2)}+\|\nabla f\|_{L_{xy}^{\infty}(\mathbb R^2)}\, ,\]

where
$$\|\nabla f\|_{L^\infty_{xy}(\mathbb R^2)}\equiv \|\partial_x f\|_{L^\infty_{xy}(\mathbb R^2)}+\|\partial_y f\|_{L^\infty_{xy}(\mathbb R^2)}.$$

It is important to point out that the existence of solutions of the IVP \eqref{BO}, as it was shown for the unidimensional BO equation by Molinet, Saut and Tzvetkov in \cite{MST2001}, can not be proved by an iteration scheme based on the Duhamel's formula. By this reason we will use the Bona-Smith argument of regularization of the initial datum in order to have a set of approximate solutions, to which we apply compactness arguments to obtain the solution when one passes to the limit.\\

An essential part in the proof of the existence of solutions to the IVP \eqref{BO} is to ensure  a common time interval in which all approximate solutions are defined. This is achieved with an a priori estimate of $\|u\|_{L_T^1W_{x,y}^{1,\infty}(\mathbb R^2)}$ for sufficiently smooth solutions of the fractional two-dimensional BO equation (see Lemma \ref{ESN}), which is a consequence of the Strichartz type estimate contained in Lemma \ref{L4.11}.\\

This article is organized as follows: section 2 is devoted to the proof of the impossibility of solving the IVP \eqref{BO} by an iterative method on the integral equation corresponding to the fractional two-dimensional BO equation; in section 3 we prove the Strichartz estimate (Lemma \ref{L4.11}), which is the main ingredient in the proof of Theorem \ref{T4.1}. In section 4, using a Kato-Ponce's commutator inequality, we establish an energy estimate for enough smooth solutions of the BO equation. In section 5, based on Lemma \ref{L4.11}, we deduce an a priori estimate for the norm $\|u\|_{L_T^1W_{xy}^{1,\infty}(\mathbb R^2)}$, where $u$ is a smooth solution of the BO equation. In the last section (section 6), we prove Theorem \ref{T4.1}; this section is divided into four subsections, corresponding, respectively, to a priori estimates, uniqueness, existence and continuous dependence on the initial data.\\

Throughout the paper the letter $C$ will denote diverse constants, which may change from line to line, and whose dependence on certain parameters is clearly established in all cases.\\

Finally, let us explain some used notation. We will denote the Fourier transform and its inverse by the symbols $\text{}^{\wedge}$ and $\text{}^\vee$, respectively.\\

For $f: \mathbb R^2\times[0,T]\to\mathbb R$ (or $\mathbb C$) we have
$$\|f \|_{L^p_{T}L^q_{xy}}:=\left( \int_0^T \left(   \int_{\mathbb R^2} |f(x,y,t)|^q dxdy  \right)^{p/q} dt \right )^{1/p}.$$
When $p=\infty$ or $q=\infty$ we must do the obvious changes with \textit{essup}. In general, for a certain Banach space $X$ and $f:[0,T]\to X$, the notation $\|f\|_{L^p_T (X)}$ means
$$\left(\int_0^T \|f(t)\|^p_X dt \right)^{1/p}.$$

The space $\bigcap_{s\in\mathbb R}H^s(\mathbb R^2)$ is denoted by $H^\infty(\mathbb R^2)$.\\

For variable expressions $A$ and $B$ the notation $A\lesssim B$  and the notation $A\gtrsim B$ mean that there exists a  universal positive constant $C$ such that $A\leq CB$ and $A\geq CB$, respectively, and the notation $A\sim B$ means that there exist universal positive constants $c$ and $C$ such that $cA\leq B\leq CA$.

\section{Picard iterative scheme on the Duhamel's formula does not work}

In this section, inspired in the work \cite{MST2001}, we prove that the IVP \eqref{BO} can not be solved by a Picard iterative scheme based on the Duhamel's formula. In other words, if
\begin{align}
u(t)=U_{\alpha}(t)\phi-\int_0^tU_{\alpha}(t-t')(u(t')u_x(t'))dt'\label{DF}
\end{align}
is the integral equation corresponding to the IVP \eqref{BO} with initial datum $\phi$, where
\begin{align}
(U_{\alpha}(t)\phi)\widehat{\;}(\xi,\eta):=e^{-it(|\xi|^{\alpha}\xi+\sgn(\xi)\eta^2)}\widehat{\phi}(\xi,\eta)\,,\label{Grupo}
\end{align}
the following assertion holds.

\begin{proposition}\label{L2.1} Let $s\in\mathbb R$ and $T>0$. Then there does not exist a space $X_T$ continuously embedded in $C([0,T]; H^s(\mathbb R^2))$ such that there exists $C>0$ with
\begin{align}
\|U_{\alpha}(\cdot_t)\phi\|_{X_T}\leq C\|\phi\|_{H^s(\mathbb R^2)}\,,\quad \phi\in H^s(\mathbb R^2)\,,\label{PL}
\end{align}
and
\begin{align}
\left\|\int_0^{\cdot_t}U_{\alpha}(\cdot_t-t')(u(t')u_x(t'))dt'\right\|_{X_T}\leq C\|u\|_{X_T}^2\,,\quad u\in X_T\,.\label{PI}
\end{align}
\end{proposition}
Let us observe that the conditions \eqref{PL} and \eqref{PI} are necessary to apply the Picard iterative scheme in the integral equation \eqref{DF}.

\begin{proof} Reasoning as in \cite{MST2001} we suppose that there exists a space $X_T$ such that \eqref{PL} and \eqref{PI} take place. Let us define in \eqref{PI} $u:=U_{\alpha}(\cdot_t)\phi$. Then
$$\left\|\int_0^{\cdot_t} U_{\alpha}(\cdot_t-t')[(U_{\alpha}(t')\phi)(U_{\alpha}(t')\phi_x)]dt'\right\|_{X_T}\leq C\|U_{\alpha}(\cdot_t)\phi\|_{X_T}^2.$$
In this way, using \eqref{PL} and the fact that $X_T\hookrightarrow C([0,T];H^s(\mathbb R))$, we would conclude that for each $t\in[0,T]$
\begin{align}
\left\|\int_0^t U_{\alpha}(t-t')[(U_{\alpha}(t')\phi)(U_{\alpha}(t')\phi_x)]dt'\right\|_{H^s(\mathbb R^2)}\lesssim \|\phi\|^2_{H^s(\mathbb R^2)},\quad \forall \phi\in H^s(\mathbb R^2).\label{oct04_04}
\end{align}
We will show that \eqref{oct04_04} is not true, exhibiting a bounded sequence $\{\phi_N\}_{N\in\mathbb N}$ in $H^s(\mathbb R^2)$, such that for $t\in[T/2,T]$,
$$\lim_{N\to\infty} \left\|\int_0^t U_{\alpha}(t-t')[(U_{\alpha}(t')\phi_N)(U_{\alpha}(t'){\phi_{N}}_x)]dt' \right\|_{H^s(\mathbb R^2)}=\infty.$$
Let us take $\phi\equiv \phi_N$ defined through the Fourier transform by
$$\widehat \phi(\xi,\eta)=\beta^{-1/2}\chi_{I_1}(\xi,\eta)+\beta^{-1/2}N^{-s }\chi_{I_2}(\xi,\eta),\quad N\gg1,\quad 0<\beta\ll1,$$
where $I_1:=[\beta/2,\beta]\times[0,\beta^{1/4}]$ and $I_2:=[N,N+\beta]\times[0,\beta^{1/4}]$, $\beta$ to be precised later, and $\chi_A$ denotes the characteristic function of the set A. Let us note that
\begin{align}
\notag\|\phi\|_{H^s}^2&=\int_{\mathbb R^2}(1+\xi^2+\eta^2)^s |\widehat\phi(\xi,\eta)|^2d\xi d\eta=\int_{I_1}(1+\xi^2+\eta^2)^s \beta^{-1}d\xi d\eta+\int_{I_2}(1+\xi^2+\eta^2)^s \beta^{-1}N^{-2s}d\xi d\eta\\
&\sim \int_{\beta/2}^\beta \int_0^{\beta^{1/4}}\beta^{-1}d\eta d\xi+\int_N^{N+\beta}\int_0^{\beta^{1/4}}N^{2s}\beta^{-1} N^{-2s} d\xi d\eta=\dfrac{\beta^{1/4}}2+\beta^{1/4}=\dfrac32\beta^{1/4}\sim\beta^{1/4}.\label{oct04_05}
\end{align}
On the other hand, if we define $\rho(\xi,\eta):=-(|\xi|^{\alpha}\xi+\sgn(\xi)\eta^2)$, then it can be easily seen that 
\begin{align*}
&\int_0^t U_{\alpha}(t-t')[(U_{\alpha}(t')\phi)(U_{\alpha}(t')\phi_x)]dt'\\
&=C\int_{\mathbb R^2}e^{ix\xi+iy\eta}e^{it\rho(\xi,\eta)}\xi\left[\int_{\mathbb R^2}\widehat\phi(\xi_1,\eta_1)\widehat \phi(\xi-\xi_1,\eta-\eta_1)\dfrac{e^{it(\rho(\xi_1,\eta_1)+\rho(\xi-\xi_1,\eta-\eta_1)-\rho(\xi,\eta))}-1}{\rho(\xi_1,\eta_1)+\rho(\xi-\xi_1,\eta-\eta_1)-\rho(\xi,\eta)}d\xi_1 d\eta_1\right]d\xi d\eta
\end{align*}
Let us consider the following four sets:
\begin{align*}
A_{11}\equiv A_{11}(\xi,\eta):=\{(\xi_1,\eta_1)\in I_1: (\xi-\xi_1,\eta-\eta_1)\in I_1\},\quad & A_{22}\equiv A_{22}(\xi,\eta):=\{(\xi_1,\eta_1)\in I_2: (\xi-\xi_1,\eta-\eta_1)\in I_2\},\\
A_{12}\equiv A_{12}(\xi,\eta):=\{(\xi_1,\eta_1)\in I_1: (\xi-\xi_1,\eta-\eta_1)\in I_2\},\quad & A_{21}\equiv A_{21}(\xi,\eta):=\{(\xi_1,\eta_1)\in I_2: (\xi-\xi_1,\eta-\eta_1)\in I_1\}.
\end{align*}
To obtain $\widehat \phi(\xi_1,\eta_1)\widehat \phi(\xi-\xi_1,\eta-\eta_1)\neq 0$, it is required that $(\xi_1,\eta_1)$ belongs to one of those four sets. Therefore, using the notation $\psi\equiv \psi(\xi,\eta,\xi_1,\eta_1):=\rho(\xi_1,\eta_1)+\rho(\xi-\xi_1,\eta-\eta_1)-\rho(\xi,\eta)$, we have that
\begin{align*}
\int_0^t U_{\alpha}(t-t')[(U_{\alpha}(t')\phi)(U_{\alpha}(t')\phi_x)]dt'=&C\int_{\mathbb R^2}e^{ix\xi+iy\eta}e^{it\rho(\xi,\eta)}\xi\left[\int_{A_{11}}\beta^{-1}\dfrac{e^{it\psi}-1}\psi d\xi_1 d\eta_1\right.\\
&\left.+\int_{A_{22}}\beta^{-1}N^{-2s}\dfrac{e^{it\psi}-1}\psi d\xi_1 d\eta_1+\int_{A_{12}\cup A_{21}}\beta^{-1}N^{-s}\dfrac{e^{it\psi}-1}\psi d\xi_1 d\eta_1\right] d\xi d\eta.
\end{align*}
Let us define
\begin{align*}
f_1(x,y,t):=&\dfrac C\beta \int_{\mathbb R^2} \xi e^{ix\xi+iy\eta}e^{it\rho(\xi,\eta)}\left[\int_{A_{11}}\dfrac{e^{it\psi}-1}\psi d\xi_1 d\eta_1\right] d\xi d\eta,\\
f_2(x,y,t):=&\dfrac C{\beta N^{2s}} \int_{\mathbb R^2} \xi e^{ix\xi+iy\eta}e^{it\rho(\xi,\eta)}\left[\int_{A_{22}}\dfrac{e^{it\psi}-1}\psi d\xi_1 d\eta_1\right] d\xi d\eta,\\
f_3(x,y,t):=&\dfrac C{\beta N^{s}} \int_{\mathbb R^2} \xi e^{ix\xi+iy\eta}e^{it\rho(\xi,\eta)}\left[\int_{A_{12}\cup A_{21}}\dfrac{e^{it\psi}-1}\psi d\xi_1 d\eta_1\right] d\xi d\eta.
\end{align*}
Then
$$\left\{\int_0^t U_{\alpha}(t-t')[(U_{\alpha}(t')\phi)(U_{\alpha}(t')\phi_x)]dt'\right\}(x,y)=f_1(x,y,t)+f_2(x,y,t)+f_3(x,y,t).$$
Since the supports of $[f_1(\cdot,\cdot,t)]^{\wedge}$, $[f_2(\cdot,\cdot,t)]^{\wedge}$ and $[f_3(\cdot,\cdot,t)]^{\wedge}$ are mutually disjoint, we have that
\begin{align}
\left\|\int_0^t U_{\alpha}(t-t')[(U_{\alpha}(t')\phi)(U_{\alpha}(t')\phi_x)]dt'\right\|_{H^s(\mathbb R^2)}\geq \|f_3(\cdot,\cdot,t)\|_{H^s(\mathbb R^2)}.\label{oct04_06}
\end{align}
By making the change of variables $\xi_1':=\xi-\xi_1$, $\eta_1':=\eta-\eta_1$ and taking into account that $\psi(\xi,\eta,\xi-\xi_1',\eta-\eta_1')=\psi(\xi,\eta,\xi_1',\eta_1')$, it is easy to see that
$$\int_{A_{21}} \dfrac{e^{i\psi(\xi,\eta,\xi_1,\eta_1)}-1}{\psi(\xi,\eta,\xi_1,\eta_1)}d\xi_1 d\eta_1=\int_{A_{12}} \dfrac{e^{i\psi(\xi,\eta,\xi_1',\eta_1')}-1}{\psi(\xi,\eta,\xi_1',\eta_1')}d\xi_1' d\eta_1'.$$
Therefore
$$f_3(x,y,t)=\dfrac{2C}{\beta N^s}\int_{\mathbb R^2} \xi e^{ix\xi+iy\eta+it\rho(\xi,\eta)}\left[\int_{A_{12}}\dfrac{e^{it\psi}-1}\psi d\xi_1 d\eta_1 \right] d\xi d\eta,$$
and
$$\|f_3(\cdot,\cdot,t)\|^2_{H^s(\mathbb R^2)}=\int_{\supp\widehat{f_3(\cdot,\cdot,t)}}(1+\xi^2+\eta^2)^s\dfrac {C\xi^2}{\beta^2 N^{2s}}\left| \int_{A_{12}}\dfrac{e^{it\psi}-1}\psi d\xi_1d\eta_1\right|^2 d\xi d\eta.$$
On the other hand, it can be seen that if $(\xi,\eta)\in\supp \widehat{f_3(\cdot,\cdot,t)}$, then $(\xi,\eta)\in[N+\beta/2,N+2\beta]\times[0,2\beta^{1/4}]$. In consequence 
$$\|f_3(\cdot,\cdot,t)\|^2_{H^s(\mathbb R^2)}\sim \dfrac{CN^{2s}N^2}{\beta^2 N^{2s}}\int_{\supp\widehat{f_3(\cdot,\cdot,t)}} \left| \int_{A_{12}}\dfrac{e^{it\psi}-1}\psi d\xi_1d\eta_1\right|^2 d\xi d\eta.$$
Let us observe that, for $(\xi,\eta)\in \supp \widehat{f_3(\cdot,\cdot,t)}$, if $(\xi_1,\eta_1)\in A_{12}(\xi,\eta)$, and we take $\beta$ in such a way that $\beta N^{\alpha}=N^{-\epsilon}$, with $0<\epsilon<\alpha$, (i.e. $\beta:=N^{-\alpha-\epsilon}$), it follows that
\begin{align*}
\psi(\xi,&\eta,\xi_1,\eta_1)=\rho(\xi_1,\eta_1)+\rho(\xi-\xi_1,\eta-\eta_1)-\rho(\xi,\eta)=-\xi_1^{1+\alpha}-(\xi-\xi_1)^{1+\alpha}+\xi^{1+\alpha}+2\eta_1(\eta-\eta_1)\\
&=-\xi_1^{1+\alpha}-[\xi^{1+\alpha}-(1+\alpha)\xi^{\alpha}\xi_1+o(\xi_1)]+\xi^{1+\alpha}+2\eta_1(\eta-\eta_1)\quad(\text{where}\;\lim_{\xi_1\to0}\dfrac{o(\xi_1)}{\xi_1}=0)\\
&=(1+\alpha)\xi^{\alpha}\xi_1-\xi_1^{1+\alpha}-o(\xi_1)+2\eta_1(\eta-\eta_1)=\xi_1\left[(1+\alpha)\xi^{\alpha}-\xi_1^{\alpha}-\dfrac{o(\xi_1)}{\xi_1}\right]+2\eta_1(\eta-\eta_1)\\
&\sim\beta N^{\alpha},
\end{align*}
since $2\eta_1(\eta-\eta_1)\lesssim \beta^{\frac12}=\frac1{N^{\frac12\alpha+\frac12\epsilon}}\ll N^{-\epsilon}=\beta N^{\alpha}$.
Taking into account that $\dfrac{1-\cos\theta}\theta\sim\theta$ for $0<\theta\ll1$, then, if $t\in[T/2,T]$, we have that
\begin{align*}
\|f_3(\cdot,\cdot,t)\|^2_{H^s}&\gtrsim Ct^2\dfrac{N^2}{\beta^2}\int_{\supp \widehat{f_3(\cdot,\cdot,t)}} \left|\re \int_{A_{12}} \dfrac{e^{it\psi}-1}{t\psi}d\xi_1 d\eta_1\right|^2 d\xi d\eta\\
&\gtrsim \dfrac{Ct^2N^2}{\beta^2} \int_{\supp \widehat{f_3(\cdot,\cdot,t)}} \left|\int_{A_{12}}\dfrac{1-\cos(t\psi)}{t\psi}d\xi_1 d\eta_1\right|^2 d\xi d\eta\\
&\gtrsim \dfrac{CT^2N^2}{\beta^2} \int_{\supp \widehat{f_3(\cdot,\cdot,t)}} \beta^2 N^{2\alpha} \left|\int_{A_{12}} d\xi_1 d\eta_1 \right|^2 d\xi d\eta.
\end{align*}

Let us observe that
$$\meas\left\{ (\xi,\eta)\in\supp \widehat{f_3(\cdot,\cdot,t)} :\meas A_{12}(\xi,\eta)  \sim\beta^{5/4} \right\} \sim \beta^{5/4}.$$
Hence, for $t\in[T/2,T]$,
\begin{align}
\|f_3(\cdot,\cdot,t)\|^2_{H^s(\mathbb R^2)} \gtrsim N^{2+2\alpha}(\beta^{5/4})^2 \beta^{5/4}=N^{2+2\alpha} \beta^{15/4}=N^{2+2\alpha}(N^{-\alpha-\epsilon})^{15/4}=N^{2-\frac74\alpha-15\epsilon/4}\to\infty,\label{oct04_07}
\end{align}

as $N\to\infty$, if $\dfrac{15\epsilon}4<2-\dfrac74\alpha$, i.e., if $\epsilon<\dfrac8{15}-\dfrac7{15}\alpha$.

In this manner for $\beta:=N^{-\alpha-\epsilon}$ with $0<\epsilon<\min\{\alpha, \dfrac8{15}-\dfrac7{15}\alpha\}$, from \eqref{oct04_05}, \eqref{oct04_06} and \eqref{oct04_07} it follows that the sequence $\{\phi_N\}_{N}$ is bounded in $H^s(\mathbb R^2)$ and that
$$\lim_{N\to\infty} \left\|\int_0^t U_{\alpha}(t-t')[(U_{\alpha}(t')\phi_N)(U_{\alpha}(t'){\phi_N}_x)] dt'\right\|_{H^s(\mathbb R^2)}=\infty,$$
for $t\in[T/2,T]$.
\end{proof}

\section{Linear Estimates}

The main objective of this section is to establish a Strichartz type estimate for the spatial derivatives of the solutions of the non homogeneous linear equation
\begin{align}
u_t+D_x^{\alpha} u_x +\mathcal Hu_{yy}=F\,,\label{LNH}
\end{align}
from the classical Strichartz type estimate of the group $\{U_{\alpha}(t)\}$, defined by \eqref{Grupo}, associated to the linear part of the fractional two-dimensional BO equation. The Strichartz estimate for the spatial derivatives of the solutions of \eqref{LNH} will allow us to control the norm $\|\nabla u\|_{L^1_T L^\infty_{xy}}$ for sufficiently smooth solutions $u$ of the fractional two-dimensional BO equation.\\

We have the following decay estimate in time for the group.
\begin{lemma}\label{DT} For $\alpha\in(0,1]$, it holds that
\begin{align}
\|D_x^{\frac{\alpha-1}2}U_{\alpha}(t)\phi\|_{L^{\infty}(\mathbb R^2)}\leq C|t|^{-1}\|\phi\|_{L^1(\mathbb R^2)}.\label{DT1}
\end{align}
\end{lemma} 
\begin{proof}
Without loss of generality we can suppose that $t>0$. Let us observe that
\begin{align}
(U_{\alpha}(t)\phi)(x,y)=(I_{\alpha}(\cdot,\cdot,t)\ast \phi)(x,y), \notag
\end{align}
where
$$I_{\alpha}(x,y,t)=(2\pi)^{-1}\int_{\mathbb R^2}e^{-it(|\xi|^{\alpha}\xi+\sgn(\xi)\eta^2)}e^{i(x\xi+y\eta)} d\xi d\eta\;.$$
Now we will analyze the decay properties of the oscillatory integral
\begin{align}
D_x^{\frac{\alpha-1}2}I_{\alpha}(x,y,t)=(2\pi)^{-1}\int_{\mathbb R_{\xi}}|\xi|^{\frac{\alpha-1}2} \left(\int_{\mathbb R_{\eta}}e^{-it\sgn (\xi)\eta^2}e^{iy\eta}d\eta \right)e^{i(x\xi-t|\xi|^{\alpha}\xi)}d\xi\,.\label{IO}
\end{align}
Performing in the inner integral the change of variables $\eta':=t^{\frac12}\eta$ we obtain
\begin{align*}
D_x^{\frac{\alpha-1}2}I_{\alpha}(x,y,t)&=\frac{C}{t^{\frac12}}\int_{\mathbb R_{\xi}}|\xi|^{\frac{\alpha-1}2} \left(\int_{\mathbb R_{\eta'}}e^{-i\sgn (\xi)\eta'^2}e^{i\frac{y}{t^{1/2}}\eta'}d\eta' \right)e^{i(x\xi-t|\xi|^{\alpha}\xi)}d\xi.
\end{align*}
Taking into account that
$$(e^{i\delta\eta'^2})^{\vee}{\;}(x)=C|\delta|^{-\frac12}e^{i\sgn(\delta)\pi/4}e^{i\delta^{-1}x^2}\quad\text{for}\; \delta\in\mathbb R-\{0\}\,,$$
we have that
$$\int_{\mathbb R_{\eta'}}e^{-i\sgn (\xi)\eta'^2}e^{i\frac{y}{t^{1/2}}\eta'}d\eta'=Ce^{-i\sgn(\xi)\pi/4}e^{-i\sgn(\xi)t^{-1}y^2}$$
and, in consequence
$$D_x^{\frac{\alpha-1}2}I_{\alpha}(x,y,t)=\frac{C}{t^{\frac12}}\int_{\mathbb R_{\xi}}|\xi|^{\frac{\alpha-1}2}e^{-i\sgn(\xi)\pi/4}e^{-i\sgn(\xi)t^{-1}y^2}e^{i(x\xi-t|\xi|^{\alpha}\xi)}d\xi\,.$$
Making in the last integral the change of variable $\xi':=t^{\frac1{1+\alpha}}\xi$ we obtain
$$D_x^{\frac{\alpha-1}2}I_{\alpha}(x,y,t)=\frac{C}{t}\int_{\mathbb R_{\xi'}}|\xi'|^{\frac{\alpha-1}2}e^{-i\sgn(\xi')\pi/4}e^{-i\sgn(\xi')t^{-1}y^2}e^{i\xi' \left(\frac{x}{t^{1/(1+\alpha)}} \right)}e^{-i|\xi'|^{\alpha}\xi'}d\xi'\,.$$
Let us define for $\lambda\in\mathbb R$,
$$J(\lambda):=\int_{\mathbb R^{+}}|\xi|^{\frac{\alpha-1}2}e^{i\xi\lambda}e^{-i|\xi|^{\alpha}\xi}d\xi\,.$$
By applying Lemma 2.7 in \cite{KPV1991} with $\phi(\xi):=-|\xi|^{\alpha}\xi$ we may conclude that there exists a positive constant $C$ such that for all $\lambda\in\mathbb R$
$$J(\lambda)\leq C\,.$$
Hence
$$\|D_x^{\frac{\alpha-1}2}I_{\alpha}(\cdot,\cdot,t)\|_{L^{\infty}(\mathbb R^2)}\leq\frac{C}{t}\,$$
and in consequence
$$\|D_x^{\frac{\alpha-1}2}U_{\alpha}(t)\phi\|_{L^{\infty}(\mathbb R^2)}=\|D_x^{\frac{\alpha-1}2}I_{\alpha}(\cdot,\cdot,t)\ast\phi\|_{L^{\infty}(\mathbb R^2)}\leq\frac{C}{t}\|\phi\|_{L^1(\mathbb R^2)}\,.$$
\end{proof}
\begin{Definition} Given $(q,p)\in\mathbb R^2$ we will say that $(q,p)$ is an admissible pair for the fractional two-dimensional BO equation if $2\leq q,p\leq\infty$, $q>2$ and
$$\frac1{q}+\frac1{p}=\frac12\,.$$
\end{Definition}
The following Strichartz's type estimate is a consequence of the decay estimate in time \eqref{DT1}.\\ 

\begin{proposition}\label{SE} Let $\alpha\in(0,1]$. If $(q,p)$ is an admissible pair for the fractional two-dimensional BO equation, then
\begin{align}
\|D_x^{\frac1{2q}(\alpha-1)}U_{\alpha}(\cdot_t)\phi\|_{L_T^qL_{xy}^p}\leq C\|\phi\|_{L_{xy}^2}\,. \label{SE1}
\end{align}
\end{proposition} 

\begin{proof}
We will denote by $q'$ and $p'$ the conjugated exponents of $q$ and $p$, respectively, i.e. $\frac1{q}+\frac1{q'}=1$ and $\frac1{p}+\frac1{p'}=1$.
We will prove that 
\begin{align}
\left\|\int_{-\infty}^{\infty}D_x^{\frac1{q}(\alpha-1)}U_{\alpha}(t-t')g(\cdot,t')dt' \right\|_{L_t^qL_{xy}^p}\leq C\|g\|_{L_T^{q'}L_{xy}^{p'}}\,.\label{ISE}
\end{align}
Let us recall that for $t\in\mathbb R$
\begin{align}
\|U_{\alpha}(t)\phi\|_{L^2(\mathbb R^2)}=\|\phi\|_{L^2(\mathbb R^2)}\,.\label{G}
\end{align}
By applying an interpolation result, similar to that in Proposition 2.3 in \cite{KPV1989}, from estimates \eqref{DT1} and \eqref{G} we obtain for $\theta\in[0,1]$ that
\begin{align}
\|D_x^{\theta(\frac{\alpha-1}2)}U_{\alpha}(t)\phi\|_{L_{xy}^{\frac2{1-\theta}}}\leq C|t|^{-\theta}\|\phi\|_{L_{xy}^{\frac2{1+\theta}}}\,.\label{INTER}
\end{align}
In this manner for $p:=\frac2{1-\theta}$, with $\theta\in(0,1)$, by Minkowski's integral inequality and estimate \eqref{INTER} we have that
\begin{align}
\notag \Big\| \int_{-\infty}^{\infty}D_x^{\theta (\frac{\alpha-1}2)}U_{\alpha}(t-t')&g(\cdot,t')dt' \Big\|_{L_t^qL_{xy}^p}\\ 
\notag&\leq\left(\int_{-\infty}^{\infty}\left(\int_{-\infty}^{\infty}\frac{C}{|t-t'|^{\theta}}\|g(\cdot,t')\|_{L_{xy}^{p'}}dt'\right)^qdt\right)^{\frac1{q}}\\
&= \left\|\int_{-\infty}^{\infty}\frac{C}{|t-t'|^{\theta}} \|g(\cdot,t')\|_{L_{xy}^{p'}}dt' \right\|_{L_t^q}=\left\|\int_{-\infty}^{\infty}\frac{C}{|t-t'|^{1-\beta}}\|g(\cdot,t')\|_{L_{xy}^{p'}}dt' \right\|_{L_t^q}\,,\label{INTER1}
\end{align}
where $\beta:=1-\theta\in(0,1)$.
Using the Hardy-Littlewood-Sobolev theorem (see \cite{LP2015}), if $\frac1{q}=\frac1{q'}-\beta$, which means $q=\frac2{\theta}$ and $\frac1{q}+\frac1{p}=\frac12$, it follows that
\begin{align}
\left\|\int_{-\infty}^{\infty}\frac{C}{|t-t'|^{1-\beta}}\|g(\cdot,t')\|_{L_{xy}^{p'}}dt' \right\|_{L_t^q}\leq C\left\| \|g(\cdot,t')\|_{L_{xy}^{p'}} \right\|_{L_t^{q'}}\,.\label{INTER2}
\end{align}
Taking into account that $q=\frac2{\theta}$, from \eqref{INTER1} and \eqref{INTER2} it follows estimate \eqref{ISE}. Finally, using the Stein-Tomas argument as in the proof of Lemma 2.4 in \cite{KPV1989}, from estimate \eqref{ISE} it follows estimate \eqref{SE1}.
\end{proof}

Let us observe that $(2,\infty)$ is not an admissible pair. We need to lose a little bit of regularity in both $x$ and $y$ directions in order to control the norm $L_T^2L_{xy}^{\infty}$.

\begin{corollary}\label{C4.10} For  $\alpha\in(0,1]$, $T>0$ and $\delta>0$ small enough, there exist $\tilde k_\delta\in(0,\frac12)$ and $c_{\delta}>0$ such that
\begin{align}
\|U_{\alpha}(\cdot_t)\phi\|_{L^2_T L^\infty_{xy}}\leq c_{\delta} T^{\tilde k_\delta} \left(\|\phi\|_{L_{xy}^2}+\|D_x^{\frac14(1-\alpha)+\delta}\phi\|_{L_{xy}^2}+\|D_y^{\delta}\phi\|_{L_{xy}^2}+\|D_x^{\frac14(1-\alpha)}D_y^{\delta}\phi\|_{L_{xy}^2} \right).\label{lessreg}
\end{align}
\end{corollary}

\begin{proof} Let us recall the Sobolev embedding (see \cite{T2006}, page 336)
\begin{align}
W_{x,y}^{\delta,p}(\mathbb R^2)\hookrightarrow L^\infty_{x,y}(\mathbb R^2) \label{embedd}
\end{align}
if $1<p<\infty$ and $\delta>0$ is such that $\delta>2/p$, where $W_{x,y}^{\delta,p}(\mathbb R^2)$ denotes the Sobolev space given by the closure of the Schwartz functions under the norm
\[\|f\|_{W_{x,y}^{\delta,p}(\mathbb R^2)}:=\|J^{\delta}f\|_{L_{xy}^p(\mathbb R^2)}\,.\]
Here $J^{\delta}$ is the operator with symbol $(1+\xi^2+\eta^2)^{\delta/2}$. Thus, if we take $\delta>0$ and $p$ such that $\delta>2/p$ and then we consider the admissible pair $(q,p)$, we have that
\begin{align*}
\|U_{\alpha}(\cdot_t)\phi\|^2_{L^2_T L^\infty_{xy}}= \int_0^T\|U_{\alpha}(t)\phi\|^2_{L^\infty_{xy}}dt\lesssim\int_0^T \|U_{\alpha}(t)\phi\|^2_{W_{xy}^{\delta,p}}dt=\int_0^T \|J^\delta U_{\alpha}(t)\phi\|^2_{L^p_{xy}} dt.
\end{align*}
Taking $\tilde q:=\dfrac q2$  and ${\tilde q}'$ such that $\dfrac1{\tilde q}+\dfrac1{{\tilde q}'}=1$, we have that
$$\dfrac1{{\tilde q}'}=1-\dfrac 1{\tilde q}=1-\dfrac 2q=\dfrac{q-2}q.$$
Using Hölder's inequality in $t$ with $\tilde q$ and ${\tilde q}'$, we obtain

\begin{align*}
\|U_{\alpha}(\cdot_t)\phi\|^2_{L^2_T L^\infty_{xy}}\lesssim \left(\int_0^T 1^{\tilde q'} dt\right)^{1/{\tilde q'}} \left(\int_0^T\|J^\delta U_{\alpha}(t) \phi\|^{2\tilde q}_{L^p_{xy}}dt \right)^{1/\tilde q}\lesssim T^{(q-2)/q}\left(\int_0^T \|J^\delta U_{\alpha}(t)\phi\|^q_{L^p_{xy}}dt \right)^{2/q}.
\end{align*}

Thus, from Proposition \ref{SE} it follows that
\begin{align}
\notag\|U_{\alpha}(\cdot_t)\phi\|_{L^2_T L^\infty_{xy}}&\lesssim T^{(q-2)/2q} \|U_{\alpha}(\cdot_t) J^\delta \phi\|_{L^q_T L^p_{xy}}=T^{(q-2)/2q} \|D_x^{\frac1{2q}(\alpha-1)}U_{\alpha}(\cdot_t)J^\delta D_x^{\frac1{2q}(1-\alpha)} \phi\|_{L^q_T L^p_{xy}}\\
 &\lesssim T^ {(q-2)/2q}\|J^{\delta}D_x^{\frac1{2q}(1-\alpha)} \phi\|_{L_{xy}^2}=T^{\tilde k_\delta}\|J^{\delta}D_x^{\frac1{2q}(1-\alpha)} \phi\|_{L_{xy}^2},\label{cuatro}
\end{align}
where $\tilde k_\delta=\dfrac{q-2}{2q}\in(0,1/2)$.\\

If $\alpha=1$ estimate \eqref{lessreg} follows from \eqref{cuatro}. Let us suppose then that $\alpha\in(0,1)$. From \eqref{cuatro}, taking into account that $\frac1{2q}(1-\alpha)<\frac14(1-\alpha)$, we can conclude, for $\delta\leq\frac14(1-\delta)$, that
\begin{align*}
\|U_{\alpha}(\cdot_t)&\phi\|_{L^2_T L^\infty_{xy}}\lesssim T^{\tilde k_\delta}\left(\|D_x^{\frac1{2q}(1-\alpha)} \phi\|_{L_{xy}^2}+\|D_x^{\delta}D_x^{\frac1{2q}(1-\alpha)} \phi\|_{L_{xy}^2}+\|D_y^{\delta}D_x^{\frac1{2q}(1-\alpha)} \phi\|_{L_{xy}^2}\right)\\
&\lesssim T^{\tilde k_\delta}\left(\|\phi\|_{L_{xy}^2}+\|D_x^{\frac1{4}(1-\alpha)} \phi\|_{L_{xy}^2}+\|D_x^{\delta}\phi\|_{L_{xy}^2}+\|D_x^{\frac1{4}(1-\alpha)}D_x^{\delta} \phi\|_{L_{xy}^2}+\|D_y^{\delta}\phi\|_{L_{xy}^2}+\|D_x^{\frac1{4}(1-\alpha)} D_y^{\delta}\phi\|_{L_{xy}^2}\right)\\
&\lesssim T^{\tilde k_\delta} (\|\phi\|_{L_{xy}^2}+\|D_x^{\frac14(1-\alpha)+\delta}\phi\|_{L_{xy}^2}+\|D_y^{\delta}\phi\|_{L_{xy}^2}+\|D_x^{\frac14(1-\alpha)}D_y^{\delta}\phi\|_{L_{xy}^2})\,,
\end{align*}
and estimate \eqref{lessreg} holds.
\end{proof}

\begin{lemma}\label{L4.11} Let $\alpha\in(0,1]$, $s_{\alpha}:=\frac32+\frac14(1-\alpha)$, $\delta>0$ small enough and $T>0$. Suppose that $w$ is a smooth solution of the linear problem
\begin{align}
u_t+D_x^{\alpha}u_x+\mathcal Hu_{yy}=F.\label{lineareq}
\end{align}
Then, there exist $k_\delta\in(1/2,1)$ and $c_\delta>0$ such that
\begin{align}
\|\partial_x w\|_{L^1_T L^\infty_{xy}}\leq&c_\delta T^{k_\delta}\left(\|w\|_{L^\infty_T H^{s_{\alpha}+2\delta}_{xy}}+\int_0^T \|F(\cdot,t)\|_{H^{s_{\alpha}-1+2\delta}_{xy}} dt \right),\label{4.25}\\
\|\partial_y w\|_{L^1_T L^\infty_{xy}}\leq&c_\delta T^{k_\delta}\left(\|w\|_{L^\infty_T H^{s_{\alpha}+2\delta}_{xy}}+\int_0^T \|F(\cdot,t)\|_{H^{s_{\alpha}-1+2\delta}_{xy}} dt \right).\label{4.25b}
\end{align}
\end{lemma}

\begin{proof} We prove \eqref{4.25}, being the proof of \eqref{4.25b} analogous. In order to do this, we use the argument in \cite{Ke2004}. First, we use a Littlewood-Paley decomposition of $\widehat{w}$ in $(\xi,\eta)$. That is, let $\varphi \in C_0^\infty(1/2<|(\xi,\eta)|<2)$ and $\chi\in C_0^\infty(|(\xi,\eta)|<2)$ such that
$$\sum_{k=1}^\infty \varphi\left( \dfrac{\xi}{2^k},\dfrac{\eta}{2^k}\right)+\chi(\xi,\eta)=1.$$
For $\Lambda:=2^k$, $k\geq 1$, define $w_\Lambda\equiv \Delta_\Lambda w:=Q_k w$, where
$$\widehat{Q_k w}(\xi,\eta)=\varphi\left( \dfrac{\xi}{2^k},\dfrac{\eta}{2^k}\right)\widehat w(\xi,\eta),$$
$w_0\equiv\Delta_0w:=Q_0 w$, and
$$\widehat{Q_0 w}(\xi,\eta)=\chi(\xi,\eta )\widehat w(\xi,\eta).$$
Let $\tilde{w}:=\displaystyle{\sum_{k\geq 1}} Q_k w$, then $w=\tilde w+w_0$.

Let us estimate $\|\partial_x w_0\|_{L^1_T L^\infty_{xy}}$. Taking into account that $w_0$ is a solution of the integral equation
\begin{align}
\partial_x w_0(t)=U_{\alpha}(t)\partial_x Q_0 w(0)+\int_0^t U_{\alpha}(t-t')\partial_x Q_0 F(\cdot,t') dt',\label{inteq}
\end{align}
we have that
\begin{align}
\notag\|\partial_x w_0\|_{L^1_T L^\infty_{xy}}&\lesssim \|U_{\alpha}(\cdot_t)\partial_x Q_0 w(0)\|_{L^1_T L^\infty_{xy}}+\int_0^T \|U_{\alpha}(\cdot_t-t')\partial_x Q_0 F(\cdot,t')\|_{L^1_T L^\infty_{xy}} dt'\\
&\lesssim T^{1/2} \left( \|U_{\alpha}(\cdot_t)\partial_x Q_0 w(0)\|_{L^2_T L^\infty_{xy}}+\int_0^T \|U_{\alpha}(\cdot_t-t')\partial_x Q_0 F(\cdot,t')\|_{L^2_T L^\infty_{xy}} dt'\right).\label{dobestrella}
\end{align}
Let us estimate $\|U_{\alpha}(\cdot_t)\partial_x Q_0 w(0)\|_{L^2_T L^\infty_{xy}}$. Define $\widetilde \chi\in C_0^\infty(\mathbb R^2)$ such that $\supp \widetilde \chi \subset\{(\xi,\eta):|(\xi,\eta)|<4\}$ and $\widetilde \chi\equiv 1$ in $\supp \chi$. Then
\begin{align*}
i\xi \widetilde \chi(\xi,\eta) e^{-it(\xi|\xi|^{\alpha}+ \sgn(\xi)\eta^2)}\chi(\xi,\eta) \widehat {w(0)}(\xi,\eta)&=i\xi e^{-it(\xi|\xi|^{\alpha}+ \sgn(\xi)\eta^2)}\chi(\xi,\eta)\widehat {w(0)}(\xi,\eta)=[U_{\alpha}(\cdot_t)\partial_x Q_0 w(0)]^{\wedge}(\xi,\eta).
\end{align*}
Hence
\begin{align*}
(i\xi\tilde\chi(\xi,\eta))^\vee \ast (e^{-it(\xi|\xi|^{\alpha}+ \sgn(\xi)\eta^2)}\chi(\xi,\eta) \widehat {w(0)}(\xi,\eta))^{\vee}=U_{\alpha}(\cdot_t)\partial_x Q_0 w(0),
\end{align*}
and
\begin{align}
\|U_{\alpha}(t)\partial_x Q_0 w(0)\|_{L^\infty_{xy}}\lesssim \|(i\xi \tilde\chi (\xi,\eta))^\vee \|_{L^1_{xy}} \|U_{\alpha}(t)Q_0 w(0)\|_{L^\infty_{xy}}\leq C \|U_{\alpha}(t) Q_0 w(0)\|_{L^\infty_{xy}}.\label{estrella}
\end{align}
From \eqref{dobestrella}, \eqref{estrella} and \eqref{lessreg} we conclude that
\begin{align}
\notag \|&\partial_x w_0\|_{L^1_T L^\infty_{xy}}\lesssim T^{1/2} \left(\|U_{\alpha}(\cdot_t)Q_0 w(0)\|_{L^2_T L^\infty_{xy}}+\int_0^T \|U_{\alpha}(\cdot_t-t') Q_0 F(\cdot, t')\|_{L^2_T L^\infty_{xy}} dt' \right)\\
\notag&\lesssim T^{1/2+\tilde k_\delta} \left( \|Q_0 w(0)\|_{L^2}+\|D_x^{\frac14(1-\alpha)+\delta}Q_0w(0)\|_{L^2}+\|D_y^{\delta}Q_0w(0)\|_{L^2}+\|D_x^{\frac14(1-\alpha)}D_y^{\delta}Q_0w(0)\|_{L^2}\right)\\
\notag&+T^{1/2+\tilde k_\delta}\int_0^T \left( \|Q_0 F(\cdot,t')\|_{L^2}+\|D_x^{\frac14(1-\alpha)+\delta}Q_0F(\cdot,t')\|_{L^2}+\|D_y^{\delta}Q_0F(\cdot,t')\|_{L^2}+\|D_x^{\frac14(1-\alpha)}D_y^{\delta}Q_0F(\cdot,t')\|_{L^2}\right) dt' \\
&\lesssim T^{1/2+\tilde k_\delta} \left( \|w(0)\|_{H^{s_{\alpha}+2\delta}(\mathbb R^2)}+\int_0^T \|F(\cdot,t')\|_{H^{s_{\alpha}-1+2\delta}(\mathbb R^2)} dt'\right).\label{4.27}
\end{align}

Now we estimate $\|\partial_x w_\Lambda\|_{L^1_T L^\infty_{xy}}$, where $\Lambda=2^k$, $k\geq 1$. To this purpose we split $[0,T]=\bigcup_j I_j$, where $I_j=[a_j,b_j]$ and $b_j-a_j=T/\Lambda$, for $j=1,\dots,\Lambda$. Define $\widetilde\varphi\in C^\infty_0(\mathbb R^2)$ such that $\supp \widetilde\varphi\subset\{(\xi,\eta):1/4<|(\xi,\eta)|<4\}$ and $\widetilde\varphi\equiv 1$ in $\supp \varphi$. Thus
$$(\partial_x w_\Lambda)^{\wedge}(\xi,\eta)=i\xi \varphi\left(\dfrac\xi{2^k},\dfrac\eta{2^k} \right)\widehat w(\xi,\eta)=i\xi\widetilde\varphi\left(\dfrac\xi{2^k},\dfrac\eta{2^k}\right)\varphi\left(\dfrac\xi{2^k},\dfrac\eta{2^k}\right)\widehat w(\xi,\eta).$$
This way,
$$\partial_x w_\Lambda=\left(i\xi \widetilde \varphi\left(\dfrac\xi{2^k},\dfrac\eta{2^k}\right) \right)^\vee \ast w_\Lambda$$
and
$$\|\partial_x w_\Lambda\|_{L^\infty_{xy}}\lesssim \left\|\left(i\xi \tilde \varphi\left(\dfrac\xi{2^k},\dfrac\eta{2^k}\right) \right)^\vee\right\|_{L^1_{xy}}\|w_\Lambda\|_{L^\infty_{xy}}.$$
Let us observe that
\begin{align*}
\left\|\left(i\xi \tilde \varphi\left(\dfrac\xi{2^k},\dfrac\eta{2^k}\right) \right)^\vee\right\|_{L^1_{xy}}&=\int_{\mathbb R_{xy}} \left| c\int_{\mathbb R_{\xi\eta}} e^{i(x\xi+y\eta)}i\xi \tilde \varphi\left(\dfrac\xi{2^k},\dfrac\eta{2^k}\right) d\xi d\eta  \right| dx dy\\
&=\int_{\mathbb R_{xy}} c\left| \int_{\mathbb R_{\xi'\eta'}} e^{i(2^kx\xi'+2^k y \eta')} i2^k \xi' \tilde\varphi(\xi',\eta') 2^{2k} d\xi' d\eta' \right| dx dy\\
&= 2^k \int_{\mathbb R_{x'y'}} c \left| \int_{\mathbb R_{\xi'\eta'}}  e^{i(x'\xi'+y'\eta')} i\xi' \tilde \varphi(\xi',\eta') d\xi' d\eta'\right| dx' dy'=C 2^k=C\Lambda.
\end{align*}
Hence $\|\partial_x w_\Lambda\|_{L^\infty_{xy}}\lesssim C\Lambda\|w_{\Lambda}\|_{L^\infty_{xy}}$. In consequence, using Cauchy-Schwarz inequality, we obtain
\begin{align}
\notag\|\partial_x w_{\Lambda}\|_{L^1_T L^\infty_{xy}}&= \sum_j \|\partial_x w_\Lambda\|_{L^1_{I_j}L^\infty_{xy}}\lesssim \Lambda \sum_j \|w_\Lambda\|_{L^1_{I_j}L^\infty_{xy}}\\
&\leq \Lambda \sum_j (b_j-a_j)^{1/2} \|w_{\Lambda}\|_{L^2_{I_j}L^\infty_{xy}}\leq (T\Lambda)^{1/2} \sum_j \|w_\Lambda\|_{L^2_{I_j}L^\infty_{xy}}.\label{4.28}
\end{align}
Next, using Duhamel's formula in each $I_j$, we obtain, for $t\in I_j$,
\begin{align}
w_\Lambda(t)=U_{\alpha}(t-a_j)w_\Lambda(a_j)+\int_{a_j}^t U_{\alpha}(t-t') F_\Lambda(t') dt'.\label{4.29}
\end{align}
Thus combining \eqref{4.28} and \eqref{4.29} and taking into account \eqref{lessreg}, we deduce that
\begin{align*}
\|\partial_x &w_\Lambda\|_{L^1_T L^\infty_{xy}}\lesssim (T\Lambda)^{1/2} \sum_j \Big( \|U_{\alpha}(\cdot_t-a_j)w_\Lambda(a_j)\|_{L^2_{I_j}L^\infty_{xy}}+ \Big\| \int_{a_j}^t U_{\alpha}(\cdot_t-t') F_\Lambda(t')dt' \Big\|_{L^2_{I_j}L^\infty_{xy}} \Big)\\
&\lesssim \Lambda^{1/2}T^{1/2+\tilde k_\delta}\sum_j \left(\|w_\Lambda(a_j)\|_{L^2}+\|D_x^{\frac14(1-\alpha)+\delta}w_\Lambda(a_j)\|_{L^2}+\|D_y^{\delta}w_\Lambda(a_j)\|_{L^2}+\|D_x^{\frac14(1-\alpha)}D_y^{\delta}w_\Lambda(a_j)\|_{L^2}\right)\\
&+\Lambda^{1/2}T^{1/2+\tilde k_\delta}\sum_j \int_{I_j}\left(\|F_\Lambda(t')\|_{L^2}+\|D_x^{\frac14(1-\alpha)+\delta}F_\Lambda(t')\|_{L^2}+\|D_y^{\delta}F_\Lambda(t')\|_{L^2}+\|D_x^{\frac14(1-\alpha)}D_y^{\delta}F_\Lambda(t')\|_{L^2} \right) dt'.
\end{align*}
Let us define $k_\delta:=1/2+\tilde k_\delta\in(1/2,1)$. We observe that in the support of $(w_\Lambda(a_j))^\wedge$ and $(F_\Lambda(t'))^\wedge$, $|(\xi,\eta)|\sim 2^k=\Lambda$, $\Lambda^{1/2}\|w_\Lambda(a_j)\|_{L^2}\sim\|w_\Lambda(a_j)\|_{H^{1/2}}$, $\Lambda^{1/2}\|D_x^{\frac14(1-\alpha)+\delta}w_\Lambda(a_j)\|_{L^2}\sim\|D_x^{\frac14(1-\alpha)+\delta}w_\Lambda(a_j)\|_{H^{1/2}}$, $\Lambda^{1/2}\|D_y^{\delta}w_\Lambda(a_j)\|_{L^2}\sim\|D_y^{\delta}w_\Lambda(a_j)\|_{H^{1/2}}$ and, $\Lambda^{1/2}\|D_x^{\frac14(1-\alpha)}D_y^{\delta}w_\Lambda(a_j)\|_{L^2}\sim\|D_x^{\frac14(1-\alpha)}D_y^{\delta}w_\Lambda(a_j)\|_{H^{1/2}}$. Therefore
\begin{align}
\notag \|\partial_x& w_\Lambda\|_{L^1_T L^\infty_{xy}}\lesssim T^{k_\delta}\sum_j \left(\|w_\Lambda(a_j)\|_{H^{1/2}}+\|D_x^{\frac14(1-\alpha)+\delta}w_\Lambda(a_j)\|_{H^{1/2}}+\|D_y^{\delta}w_\Lambda(a_j)\|_{H^{1/2}}+\|D_x^{\frac14(1-\alpha)}D_y^{\delta}w_\Lambda(a_j)\|_{H^{1/2}}\right)\\
\notag&+T^{ k_\delta}\sum_j \int_{I_j}\left(\|F_\Lambda(t')\|_{H^{1/2}}+\|D_x^{\frac14(1-\alpha)+\delta}F_\Lambda(t')\|_{H^{1/2}}+\|D_y^{\delta}F_\Lambda(t')\|_{H^{1/2}}+\|D_x^{\frac14(1-\alpha)}D_y^{\delta}F_\Lambda(t')\|_{H^{1/2}} \right) dt'\\
\notag&\lesssim T^{k_\delta}\sum_j \Lambda^{-1}\left(\|w_\Lambda(a_j)\|_{H^{3/2}}+\|D_x^{\frac14(1-\alpha)+\delta}w_\Lambda(a_j)\|_{H^{3/2}}+\|D_y^{\delta}w_\Lambda(a_j)\|_{H^{3/2}}+\|D_x^{\frac14(1-\alpha)}D_y^{\delta}w_\Lambda(a_j)\|_{H^{3/2}}\right)\\
\notag&+T^{ k_\delta}\int_0^T\left(\|F_\Lambda(t')\|_{H^{1/2}}+\|D_x^{\frac14(1-\alpha)+\delta}F_\Lambda(t')\|_{H^{1/2}}+\|D_y^{\delta}F_\Lambda(t')\|_{H^{1/2}}+\|D_x^{\frac14(1-\alpha)}D_y^{\delta}F_\Lambda(t')\|_{H^{1/2}} \right) dt'\\
\notag&\lesssim T^{k_\delta} \sup_{t\in[0,T]}\left(\|w_\Lambda(t)\|_{H^{3/2}}+ \|D_x^{\frac14(1-\alpha)+\delta}w_\Lambda(t)\|_{H^{3/2}}+\|D_y^{\delta}w_\Lambda(t)\|_{H^{3/2}}+\|D_x^{\frac14(1-\alpha)}D_y^{\delta}w_\Lambda(t)\|_{H^{3/2}}\right)\big(\sum_j \Lambda^{-1}\big)\\
&+T^{k_\delta}\int_0^T \left(\|F_\Lambda(t')\|_{H^{1/2}}+\|D_x^{\frac14(1-\alpha)+\delta}F_\Lambda(t')\|_{H^{1/2}}+\|D_y^{\delta}F_\Lambda(t')\|_{H^{1/2}}+\|D_x^{\frac14(1-\alpha)}D_y^{\delta}F_\Lambda(t')\|_{H^{1/2}} \right)dt'.\label{4.30}
\end{align}
Hence
\begin{align}
\notag &\|\partial_x \tilde w\|_{L^1_T L^\infty_{xy}}\lesssim \sum_{k\geq 1} \|\partial_x Q_k w\|_{L^1_T L^\infty_{xy}}\\
\notag&\lesssim  T^{k_\delta} \sum_{k\geq1}\sup_{t\in[0,T]}(\|Q_kw(t)\|_{H^{3/2}}+ \|D_x^{\frac14(1-\alpha)+\delta}Q_kw(t)\|_{H^{3/2}}+\|D_y^{\delta}Q_kw(t)\|_{H^{3/2}}+\|D_x^{\frac14(1-\alpha)}D_y^{\delta}Q_kw(t)\|_{H^{3/2}})\\
\notag&+T^{k_\delta}\sum_{k\geq1}\int_0^T (\|Q_kF(t')\|_{H^{1/2}}+\|D_x^{\frac14(1-\alpha)+\delta}Q_kF(t')\|_{H^{1/2}}+\|D_y^{\delta}Q_kF(t')\|_{H^{1/2}}+\|D_x^{\frac14(1-\alpha)}D_y^{\delta}Q_kF(t')\|_{H^{1/2}} )dt'\\
\notag &\leq T^{k_\delta} \sum_{k\geq 1} 2^{-k\delta} \sup_{t\in[0,T]} (\|Q_kw(t)\|_{H^{3/2+\delta}}+ \|D_x^{\frac14(1-\alpha)+\delta}Q_kw(t)\|_{H^{3/2+\delta}}+\|D_y^{\delta}Q_kw(t)\|_{H^{3/2+\delta}}\\
\notag&+\|D_x^{\frac14(1-\alpha)}D_y^{\delta}Q_kw(t)\|_{H^{3/2+\delta}}) \\
\notag&+T^{k_\delta}\sum_{k\geq1}2^{-k\delta}\int_0^T (\|Q_kF(t')\|_{H^{1/2+\delta}}+\|D_x^{\frac14(1-\alpha)+\delta}Q_kF(t')\|_{H^{1/2+\delta}}+\|D_y^{\delta}Q_kF(t')\|_{H^{1/2+\delta}}\\
\notag&+\|D_x^{\frac14(1-\alpha)}D_y^{\delta}Q_kF(t')\|_{H^{1/2+\delta}} )dt'\\
\notag &\lesssim T^{k_\delta} \sup_{t\in[0,T]} (\|w(t)\|_{H^{3/2+\delta}}+ \|D_x^{\frac14(1-\alpha)+\delta}w(t)\|_{H^{3/2+\delta}}+\|D_y^{\delta}w(t)\|_{H^{3/2+\delta}}+\|D_x^{\frac14(1-\alpha)}D_y^{\delta}w(t)\|_{H^{3/2+\delta}})  (\sum_{k\geq 1} 2^{-k\delta} ))  \\
\notag&+T^{k_\delta}\int_0^T (\|F(t')\|_{H^{1/2+\delta}}+\|D_x^{\frac14(1-\alpha)+\delta}F(t')\|_{H^{1/2+\delta}}+\|D_y^{\delta}F(t')\|_{H^{1/2+\delta}}\\
\notag&+\|D_x^{\frac14(1-\alpha)}D_y^{\delta}F(t')\|_{H^{1/2+\delta}} )dt'(\sum_{k\geq1}2^{-k\delta})\\
&\lesssim T^{k_\delta} ( \|w\|_{L^\infty_T H^{3/2+\frac14(1-\alpha)+2\delta}(\mathbb R^2)} +\int_0^T \|F(\cdot,t)\|_{H^{1/2+\frac14(1-\alpha)+2\delta}(\mathbb R^2)}dt).   \label{4.31}
\end{align}
Thus, from \eqref{4.27} and \eqref{4.31} it follows \eqref{4.25}.
\end{proof}

\section{Energy Estimates}

Using energy estimates on the fractional two-dimensional BO equation and the classical Kato-Ponce commutator inequality, in this section we obtain the following a priori estimate for sufficiently smooth solutions of the fractional BO equation.

\begin{lemma} Let $s>0$ and $T>0$. Let $u\in C([0,T];H^\infty(\mathbb R^2))$ be a real solution of the IVP \eqref{BO}. Then, there exists a positive constant $C_0$ such that
\begin{align}
\|u\|^2_{L^\infty_T H^s_{xy}}\leq \|u_0\|^2_{H^s_{xy}}+C_0\|\nabla u\|_{L^1_TL^\infty_{xy}} \|u\|^2_{L^\infty_T H^s_{xy}}.\label{4.34}
\end{align}
\end{lemma}

\begin{proof} First of all, let us observe that the operator  $D_x^{\alpha} \partial_x +\mathcal H\partial_{yy}$ is skew-adjoint in $L^2(\mathbb R^2)$. In fact, if we denote by $(\cdot,\cdot)$ the inner product in $L^2$ we have
\begin{align*}
(D_x^{\alpha} u_x +\mathcal Hu_{yy},v)&=((|\xi|^{\alpha}i\xi-i\sgn(\xi)i^2\eta^2)\hat{u},\hat v)=(\hat u, (-|\xi|^{\alpha}i\xi+i\sgn(\xi)i^2\eta^2)\hat v)\\
&=(u,-D_x^{\alpha}v_x-\mathcal H v_{yy})=-(u,D_x^{\alpha}v_x+\mathcal H v_{yy})\,.
\end{align*}
If we take in the last equality $u=v$ a real function, then we obtain
\begin{align}
(D_x^{\alpha} u_x +\mathcal Hu_{yy},u)=0\,.\label{SA}
\end{align}
Applying the operator $J^s\equiv (1-\Delta)^{s/2}$ to the equation in \eqref{BO}, multiplying by $J^su(t)$, integrating in $\mathbb R^2_{xy}$ and, taking into account \eqref{SA},we obtain, for $t\in[0,T]$, that
\begin{align}
\dfrac12\dfrac d{dt}\|J^s u(t)\|^2_{L^2_{xy}}+\int_{\mathbb R^2}J^s(u(t)\partial_x u(t))J^su(t) dxdy=0.\label{ag02_2}
\end{align}
Kato-Ponce's commutator inequality (see \cite{KaPo1988}) affirms that
\begin{align}
\|[J^s,f]g\|_{L^2(\mathbb R^2)}\leq C (\|\nabla f\|_{L^\infty(\mathbb R^2)}\|J^{s-1}g\|_{L^2(\mathbb R^2)}+\|J^s f\|_{L^2(\mathbb R^2)}\|g\|_{L^\infty(\mathbb R^2)})\,,\label{KP}
\end{align}
where $[J^s,f]g:=J^s(fg)-fJ^sg$.\\
Let us note that
\begin{align}
\notag \int_{\mathbb R^2} J^s(u(t)\partial_x u(t))J^s u(t)dx dy&=(J^s(u(t)\partial_x u(t)),J^s(u(t)))\\
&=([J^s,u(t)]\partial_x u(t),J^s u(t))-\dfrac12(\partial_x u(t)J^s u(t),J^s u(t)).\label{ag02_3}
\end{align}
From Kato-Ponce's commutator inequality it follows that
\begin{align}
\notag \|[J^s,u(t)]\partial_x u(t)\|_{L^2}&\leq C(\|\nabla u(t)\|_{L^\infty_{xy}}\|J^{s-1}\partial_x u(t)\|_{L^2(\mathbb R^2)}+\|J^su(t)\|_{L^2(\mathbb R^2)}\|\partial_x u(t)\|_{L^\infty(\mathbb R^2)})\\
&\leq C \|\nabla u(t)\|_{L^\infty_{xy}}\|J^s u(t)\|_{L^2(\mathbb R^2)}.\label{ag02_4}
\end{align}

In consequence, from \eqref{ag02_2} to \eqref{ag02_4} we can conclude that
\begin{align}
\dfrac d{dt}\|u(t)\|^2_{H^s}\leq C\|\nabla u(t)\|_{L^\infty_{xy}}\|u(t)\|^2_{H^s}.\label{4.35}
\end{align}
Now, we integrate \eqref{4.35} in $[0,t]$ to obtain
\begin{align*}
 \|u(t)\|^2_{H^s}&\leq \|u_0\|^2_{H^s}+C\int_0^t \|\nabla u(t')\|_{L^\infty_{xy}}\|u(t')\|^2_{H^s}dt'\leq \|u_0\|^2_{H^s}+C\|u\|^2_{L^\infty_T H^s_{xy}}\int_0^t \|\nabla u(t')\|_{L^\infty_{xy}}dt'\\
&\leq \|u_0\|^2_{H^s}+C\|\nabla u\|_{L^1_T L^\infty_{xy}} \|u\|^2_{L^\infty_T H^s_{xy}},
\end{align*}
for all $t\in[0,T]$. Then, \eqref{4.34} follows from the last inequality.
\end{proof}

\section{Estimates for the norm $\|u\|_{L_T^1W_{x,y}^{1,\infty}(\mathbb R^2)}$ }

We will derive an a priori estimate for the norm
\begin{align}
f(T):=\|u\|_{L^1_T L^\infty_{xy}}+\|\nabla u\|_{L^1_T L^\infty_{xy}}\label{4.37}
\end{align}
based on the refined Strichartz estimate proved in Lemma \ref{L4.11}. This a priori estimate is crucial to guarantee that the approximate solutions, used in the proof of the existence of solutions to the IVP \eqref{BO}, are defined in a common time interval $[0,T]$ (see Lemma \ref{L5.3}).

\begin{lemma}\label{ESN} Let $T>0$ and let $u\in C([0,T];H^\infty(\mathbb R^2))$ be a solution of the IVP \eqref{BO}. Then, for any $s>s_{\alpha}:=\frac32+\frac14(1-\alpha)$, there exist $k_s\in (1/2,1)$ and $C_s>0$ such that
\begin{align}
f(T)\leq C_s T^{k_s}(1+f(T)) \|u\|_{L^\infty_T H^s_{xy}}.\label{4.38}
\end{align}
\end{lemma}

\begin{proof}
Let us fix a constant $\delta_0$ such that $0<\delta_0<s-s_{\alpha}$. We first estimate $\|\nabla u\|_{L^1_TL^\infty_{xy}}$. Taking $F:=-u\partial_x u$ in equations \eqref{4.25} and \eqref{4.25b}, from Lemma \ref{L4.11} we get
\begin{align}
\|\nabla u\|_{L^1_T L^\infty_{xy}}\leq c_\delta T^{k_\delta}\left(\|u\|_{L^\infty_T H^{s_{\alpha}+2\delta}_{xy}}+\int_0^T \|u(t)\partial_x u(t)\|_{H^{s_{\alpha}-1+2\delta}_{xy}}dt\right),\label{4.39}
\end{align}
where $\delta$ is such that $0<\delta<\delta_0$ and will be determined during the proof.\\

Let us bound each one of the terms in the right hand side of \eqref{4.39}. 

By choosing $\delta>0$ such that $0<\delta<\delta_0/2$, it is clear that $s_{\alpha}+2\delta<s_{\alpha}+\delta_0<s$ and in consequence
\begin{align}
\|u\|_{L^\infty_T H^{s_{\alpha}+2\delta}_{xy}}\leq\|u\|_{L^\infty_T H^s_{xy}}.\label{4.41}
\end{align}

On the other hand
\begin{align*}
\int_0^T \|u(t)\partial_x u(t)\|_{H^{s_{\alpha}-1+2\delta}_{xy}} dt\lesssim&\int_0^T \|u(t)\partial_x u(t)\|_{L^2_{xy}} dt+\int_0^T \|D_x^{s_{\alpha}-1+2\delta}(u(t)\partial_x u(t))\|_{L^2_{xy}}dt\\
&+\int_0^T \|D_y^{s_{\alpha}-1+2\delta}(u(t)\partial_x u(t))\|_{L^2_{xy}}dt\equiv I_1+I_2+I_3.
\end{align*}
Let us estimate $I_1$:
\begin{align}
I_1\leq \sup_{t\in[0,T]} \|u(t)\|_{L^2}\int_0^T \|\partial_x u(t)\|_{L^\infty_{xy}}dt\leq f(T) \|u\|_{L^\infty_T L^2_{xy}}\leq f(T)\|u\|_{L^\infty_T H^s_{xy}}.\label{4.45}
\end{align}
Now we estimate $I_2$:\\

 The fractional Leibniz rule, proved in \cite{KPV1993}, says that for $\sigma\in(0,1)$, it holds that
\begin{align}
\|D_x^\sigma(fg)\|_{L^2(\mathbb R)}\lesssim \|D_x^\sigma f\|_{L^{p_1}(\mathbb R)}\|g\|_{L^{q_1}(\mathbb R)}+\|D_x^\sigma g\|_{L^{p_2}(\mathbb R)}\|f\|_{L^{q_1}(\mathbb R)},\label{4.6}
\end{align}
with
$$\dfrac1{p_i}+\dfrac1{q_i}=\dfrac12,\quad 1<p_i,q_i\leq \infty,\quad i=1,2.$$
Taking in \eqref{4.6} $\sigma:=s_{\alpha}-1+2\delta\in(0,1)$ ($\delta>0$ small enough), and $p_1=p_2:=2$, $q_1=q_2:=\infty$  we have that
\begin{align}
\notag I_2=&\int_0^T \left[\int_{\mathbb R_{y}} \|D_x^{s_{\alpha}-1+2\delta}(u(t)\partial_x u(t))(\cdot_x,y)\|^2_{L^2(\mathbb R_x)} dy\right]^{1/2}dt\\
\notag \lesssim &\int_0^T  \left[ \int_{\mathbb R_y} \left(\|D_x^{s_{\alpha}-1+2\delta}\partial_xu(t)(\cdot_x,y)\|^2_{L^2(\mathbb R_x)}\|u(t)(\cdot_x,y)\|^2_{L^\infty(\mathbb R_x)}\right.\right.\\
\notag&\left.\left.\hspace{2cm}+\|\partial_x u(t)(\cdot_x,y)\|^2_{L^\infty(\mathbb R_x)}\|D_x^{s_{\alpha}-1+2\delta}u(t)(\cdot_x,y)\|^2_{L^2(\mathbb R_x)} \right)dy\right]^{1/2}dt\\
\notag\lesssim& \int_0^T \left[ \|u(t)\|^2_{L^\infty_{xy}}\|D_x^{s_{\alpha}-1+2\delta}\partial_x u(t)\|^2_{L^2_{xy}} +\|\partial_x u(t)\|^2_{L^\infty_{xy}} \|D_x^{s_{\alpha}-1+2\delta}u(t)\|^2_{L^2_{xy}}\right]^{1/2}dt\\
\notag\lesssim& \int_0^T \left[ \|u(t)\|_{L^\infty_{xy}}\|D_x^{s_{\alpha}-1+2\delta}\partial_x u(t)\|_{L^2_{xy}} +\|\partial_x u(t)\|_{L^\infty_{xy}} \|D_x^{s_{\alpha}-1+2\delta}u(t)\|_{L^2_{xy}}\right]dt\\
\lesssim & \|u\|_{L^\infty_T H_{xy}^s}\left(\|u\|_{L^1_T L^\infty_{xy}} +\|\partial_x u\|_{L^1_T L^\infty_{xy}}\right)\leq f(T)\|u\|_{L^\infty_T H^s_{xy}}.\label{4.46}
\end{align}

In a similar way, we obtain
\begin{align}
I_3\lesssim f(T)\|u\|_{L^\infty_T H^s_{xy}}.\label{4.47}
\end{align}

From \eqref{4.39}, \eqref{4.41}, \eqref{4.45}, \eqref{4.46} and \eqref{4.47} we conclude that
\begin{align}
\|\nabla u\|_{L^1_T L^\infty_{xy}}\leq C_\delta T^{k_\delta}(1+f(T))\|u\|_{L^\infty_T H^s_{xy}}.\label{4.48}
\end{align}

In order to estimate the norm $\|u\|_{L^1_T L^\infty_{xy}}$ we observe that $u$ is a solution of the integral equation
\begin{align}
u(t)=U_{\alpha}(t)u_0+\int_0^t U_{\alpha}(t-t')(u(t')\partial_xu(t')) dt'.\label{Duhamel}
\end{align}
 Then, using Hölder's inequality in time and \eqref{lessreg} we have that
\begin{align*}
\|u\|_{L^1_T L^\infty_{xy}}\leq& \|U_{\alpha}(\cdot_t)u_0\|_{L^1_T L^\infty_{xy}}+\int_0^T\|U_{\alpha}(\cdot_t-t')(u(t')\partial_x u(t'))\|_{L^1_T L^\infty_{xy}}dt'\\
\leq & T^{1/2} \|U_{\alpha}(\cdot_t)u_0\|_{L^2_T L^\infty_{xy}}+T^{1/2}\int_0^T \|U_{\alpha}(._t-t')(u(t')\partial_x u(t'))\|_{L^2_T L^\infty_{xy}} dt'\\
\leq&C_\delta T^{1/2} T^{\tilde k_\delta}\|u_0\|_{H^{1/4(1-\alpha)+\delta}(\mathbb R^2)}+C_\delta T^{1/2}T^{\tilde k_\delta}\int_0^T \|u(t')\partial_x u(t')\|_{H^{1/4(1-\alpha)+\delta}(\mathbb R^2)}dt'\\
\lesssim& T^{k_\delta}\|u_0\|_{H^s}+T^{k_\delta}\int_0^T \|u(t')\partial_x u(t')\|_{H^{1/4(1-\alpha)+\delta}(\mathbb R^2)}dt'\\
\leq&T^{k_\delta} \|u\|_{L^\infty_T H^s_{xy}}+T^{k_\delta}\int_0^T \|u(t')\partial_x u(t')\|_{H^{1/4(1-\alpha)+\delta}(\mathbb R^2)}dt'\,,
\end{align*}

where $k_\delta:=\dfrac12+\tilde k_\delta\in(1/2,1)$.

Proceeding as it was done to estimate
$$\int_0^T\|u(t') \partial_x u(t')\|_{H^{s_{\alpha}-1+2\delta} }dt',$$
we can prove that
$$\int_0^T\|u(t')\partial_x u(t')\|_{H^{1/4(1-\alpha)+\delta}(\mathbb R^2)}\lesssim f(T)\|u\|_{L^\infty_T H^s_{xy}}.$$
Hence
\begin{align}
\|u\|_{L^1_T L^\infty_{xy}}\lesssim T^{k_\delta} (1+f(T))\|u\|_{L^\infty_T H^s_{xy}}.\label{4.63}
\end{align}
From \eqref{4.48} and \eqref{4.63} it follows \eqref{4.38}.
\end{proof}

\section{Proof of Theorem \ref{T4.1}}

Using the abstract theory, developed by Kato in \cite{Ka1975} and \cite{Ka1979}, to prove LWP of quasi-linear evolution equations, it can be established the following result of LWP of the IVP \eqref{BO} for initial data in $H^s(\mathbb R^2)$, with $s>2$. 

\begin{lemma}\label{L5.2} Let $s>2$ and $u_0\in H^s(\mathbb R^2)$. There exist a positive time $T=T(\|u_0\|_{H^s})$ and a unique solution of the IVP \eqref{BO} in the class
$$C([0,T];H^s(\mathbb R^2))\cap C^1([0,T]; L^2(\mathbb R^2)).$$
Moreover, for any $0<T'<T$, there exists a neighborhood $\mathcal U$ of $u_0$ in $H^s(\mathbb R^2)$ such that the flow map datum-solution
\begin{align*}
S^s_{T'}:\mathcal U&\to C([0,T'];H^s(\mathbb R^2))\\
v_0&\mapsto v,
\end{align*}
is continuous.
\end{lemma}

Next we will describe the abstract theorem of Kato  and how the IVP \eqref{BO}, considered by us, satisfies the hypotheses of this abstract theorem  when the initial data belong to $H^s(\mathbb R^2)$ with $s>2$. In other words, we will show how to prove Lemma \ref{L5.2} using Kato's theory.\\

Let $X$, $Y$ be reflexive and separable Banach spaces such that $Y\subset X$ with continuous and dense immersion. Suppose that there is a surjective isometry $S:Y\longrightarrow X$. Let $A$ be a function defined on $Y$ and with values in the set of linear operators in $X$. We say that the system $(X,Y,A,S)$ is admissible for the equation
\begin{align}
\frac{du}{dt}+A(u(t))u(t)=0\,,\label{QE}
\end{align}
if, for $R>0$, there are positive constants $\beta(R)$, $\lambda(R)$, $\mu(R)$, $\nu(R)$, $\rho(R)$, called the parameters of the system, such that, if $\|v\|_Y<R$, and $\|w\|_Y<R$, then:\\

(i) $-A(w)$ is the generator of a $C_0$-semigroup $\{e^{-tA(w)}\}_{t\geq0}$ in $X$ satisfying
$$\|e^{-tA(w)}\|_{BL(X)}\leq e^{\beta(R)t}\quad\forall t\geq0\,;$$
(ii) $Y\subset D(A(w))$ (the domain of $A(w)$), $A(w)|_Y\in B(Y;X)$ (the space of bounded linear operators from $Y$ to $X$), and
$$\|A(w)|_Y\|_{B(Y;X)}\leq\lambda(R)\quad\text{and}\quad\|(A(w)-A(v))|_Y\|_{B(Y;X)}\leq\mu(R)\|w-v\|_X\,;$$
(iii) $SA(w)S^{-1}-A(w)$ admits an extension $B(w)\in B(X)$ which satisfies
$$\|B(w)\|_{B(X)}\leq\nu(R)\quad\text{and}\quad\|B(w)-B(v)\|_{B(X)}\leq\rho(R)\|w-v\|_Y\,.$$
\begin{Theorem}\label{Kato} (see \cite{Ka1975} and \cite{Ka1979}) If the system $(X,Y,A,S)$ is admissible for the evolution equation \ref{QE}, then, given $u_0\in Y$, there exists $T=T(\|u_0\|_Y)$ such that the IVP 
\begin{align}
\left. \begin{array}{rl}
\dfrac{du}{dt}+A(u(t))u(t)&\hspace{-2mm}=0,\\
u(0)&\hspace{-2mm}=u_0,
\end{array} \right\}\label{QE1}
\end{align}
has a unique solution $u\in C([0,T];Y)\cap C^1([0,T];X)$.\\
Besides, if $u\in C([0,T_{\max});Y)\cap C^1([0,T_{\max});X)$ is the unique noncontinuable solution of \eqref{QE1} and $0<T<T_{\max}$, then there exists a neighborhood $V$ of $u_0$ in $Y$ such that the map $\tilde{u_0}\mapsto \tilde{u}$ is continuous from $V$ to $C([0,T];Y)$.
\end{Theorem}
To IVP \eqref{BO} we associate the abstract IVP \eqref{QE1} by defining
$$A(v)u:=D_x^{\alpha} u_x +\mathcal Hu_{yy}+v u_x\,.$$
Let $s>2$ and $S:=(1-\Delta)^{s/2}\equiv J^s$. Then the system $(L^2(\mathbb R^2), H^s(\mathbb R^2), A, S)$ is admissible for the quasilinear equation in IVP \eqref{QE1}. Properties (i) and (ii) can be easily verified by using standard methods of the theory of linear semigroups. With respect to condition (iii), we can use the Kato-Ponce's commutator inequality \eqref{KP}. In fact, taking into account that $s>2$, we have that $H^{s-1}(\mathbb R^2)\hookrightarrow L^{\infty}(\mathbb R^2)$, and in consequence
\begin{align}
\notag\|(SA(w)S^{-1}-&A(w))u\|_{L^2(\mathbb R^2)}=\|[J^s,w]S^{-1}u_x\|_{L^2(\mathbb R^2)}\\
\notag&\leq C(\|\nabla w\|_{L^{\infty}(\mathbb R^2)}\|J^{s-1}S^{-1}u_x\|_{L^2(\mathbb R^2)}+\|J^sw\|_{L^2(\mathbb R^2)}\|S^{-1}u_x\|_{L^{\infty}(\mathbb R^2)})\\
\notag&\leq C(\|\nabla w\|_{H^{s-1}(\mathbb R^2)}\|u\|_{L^{2}(\mathbb R^2)}+\|w\|_{H^{s}(\mathbb R^2)}\|S^{-1}u_x\|_{H^{s-1}(\mathbb R^2)})\\
&\leq C\|w\|_{H^{s}(\mathbb R^2)}\|u\|_{L^{2}(\mathbb R^2)}\,,\label{KP1}
\end{align}
which implies (iii) with $\nu(R)=CR$ and $\rho(R)=C$, for some constant $C$. In this manner Lemma \ref{L5.2} has been proved.\\

Lemma \ref{L5.2} will allow us to apply the Bona-Smith method for proving the existence part of Theorem \ref{T4.1}.\\

\subsection{A Priori Estimates}

Let $u_0\in H^\infty(\mathbb R^2)$ and $u$ the solution of the IVP \eqref{BO} of Lemma \ref{L5.2}. Then $u\in C([0,T^*);H^\infty(\mathbb R^2))$, where $T^*$ is the maximal time of existence of $u$ satisfying $T^*\geq T(\|u_0\|_{H^3})$. We have either $T^*=+\infty$ or, if $T^*<\infty$,
\begin{align}
\lim_{t\to T^*} \|u(t)\|_{H^3}=+\infty.\label{5.2}
\end{align}

\begin{lemma}\label{L5.3} Let $s\in(s_{\alpha},3)$, $C_s$ the constant in \eqref{4.38}, $C_0$ the constant in \eqref{4.34}, $A_s:=8(1+C_0)C_s$ and $T:=(A_s\|u_0\|_{H^s}+1)^{-2}$. Then $T<T^*$, 
\[\|u\|_{L^\infty_T H^s_{xy}}\leq 2\|u_0\|_{H^s_{xy}}\quad\text{and}\quad f(T)=\|u\|_{L^1([0,T];W_{xy}^{1,\infty}(\mathbb R^2))}\leq \dfrac 83 C_s\|u_0\|_{H^s_{xy}}.\]
\end{lemma}

\begin{proof} Let $A$ be the set $\{T'\in(0,T^*):\|u\|^2_{L^\infty_{T'} H^s}\leq 4\|u_0\|^2_{H^s}\}$. Since $u\in C([0,T^*);H^\infty(\mathbb R^2))$, the set $A$ is not empty. Let $T_0:=\sup A$. We will prove that $T_0>T$. We argue by contradiction by assuming that $0<T_0\leq T<1$. By continuity we have that $\|u\|^2_{L^\infty_{T_0}H^s}\leq 4\|u_0\|^2_{H^s}$. Let $k_s\in(\dfrac12,1)$ be the exponent in \eqref{4.38}. Then, using \eqref{4.38}, we have
\begin{align*}
f(T_0)&\leq C_s T_0^{k_s}(1+f(T_0))\|u\|_{L^\infty_{T_0}H^s}\leq C_s T_0^{k_s}(1+f(T_0))2\|u_0\|_{H^s}\\
&\leq 2C_s T_0^{1/2}\|u_0\|_{H_s}(1+f(T_0))\leq 2 C_s\dfrac1{A_s\|u_0\|_{H^s}+1}\|u_0\|_{H^s}(1+f(T_0)).
\end{align*}
Hence
$$\left(1-\dfrac{2C_s\|u_0\|_{H^s}}{A_s \|u_0\|_s+1} \right) f(T_0)\leq \dfrac{2C_s\|u_0\|_{H^s}}{A_s\|u_0\|_{H^s}+1}.$$
Therefore
$$(A_s \|u_0\|_{H^s}+1-2C_s\|u_0\|_{H^s})f(T_0)\leq 2C_s\|u_0\|_{H^s},$$
i.e.,
$$(8(1+C_0)C_s\|u_0\|_{H^s}+1-2C_s\|u_0\|_{H^s})f(T_0)\leq 2C_s\|u_0\|_{H^s},$$
i.e.,
$$(6C_s\|u_0\|_{H^s}+8C_0C_s\|u_0\|_{H^s}+1)f(T_0)\leq 2C_s\|u_0\|_{H^s}.$$
In consequence
$$f(T_0)\leq \dfrac{2C_s\|u_0\|_{H^s}}{6C_s\|u_0\|_{H^s}+8C_0C_s\|u_0\|_{H^s}+1}\leq \dfrac{2C_s}{6C_s+8C_0C_s}=\dfrac1{3+4C_0}\leq \dfrac1{3C_0}.$$
We use the estimate \eqref{4.34} with $s=3$ to obtain
$$\|u\|^2_{L^\infty_{T_0} H^3}\leq \|u_0\|^2_{H^3}+C_0 \|\nabla u\|_{L^1_{T_0}L^\infty_{xy}}\|u\|^2_{L^\infty_{T_0} H^3}\leq \|u_0\|^2_{H^3}+C_0 \dfrac1{3C_0}\|u\|^2_{L^\infty_{T_0}H^3}.$$
This way,
$$\dfrac23 \|u\|^2_{L^\infty_{T_0}H^3}\leq \|u_0\|^2_{H^3},$$
i.e.,
$$\|u\|^2_{L^\infty_{T_0} H^3}\leq \dfrac32\|u_0\|^2_{H^3},$$
which implies, taking into account \eqref{5.2}, that $T_0<T^*$.\\

On the other hand, by using the energy estimate \eqref{4.34}, we have that
$$\|u\|^2_{L^\infty_{T_0}H^s}\leq \|u_0\|^2_{H^s}+C_0f(T_0)\|u\|^2_{L^\infty_{T_0}H^s}\leq \|u_0\|^2_{H^s}+\dfrac13\|u\|^2_{L^\infty_{T_0}H^s},$$
i.e.,
$$\|u\|^2_{L^\infty_{T_0}H^s}\leq\dfrac32\|u_0\|^2_{H^s},$$
and by continuity, for some $T'\in(T_0,T^*)$, we have that
$$\|u\|^2_{L^\infty_{T'}H^s}\leq 4\|u_0\|^2_{H^s},$$
i.e., there exists $T'>T_0$, with $T'\in A$. This contradicts the definition of $T_0$. Then we conclude that $T<T_0$ and thus $\|u\|_{L^\infty_{T}H^s}\leq 2\|u_0\|_{H^s}$.\\

Now we prove that $f(T)\leq \dfrac 83 C_s \|u_0\|_{H^s}$. From \eqref{4.38} it follows that
$$f(T)\leq C_s T^{k_s}(1+f(T))\|u\|_{L^\infty_{T}H^s}\leq 2C_sT^{k_s}(1+f(T))\|u_0\|_{H^s},$$
hence
$$(1-2C_s T^{k_s}\|u_0\|_{H^s})f(T)\leq 2C_s T^{k_s}\|u_0\|_{H^s}.$$
Let us observe that
\begin{align*}
2C_sT^{k_s}\|u_0\|_{H^s}=2C_s\dfrac1{(A_s\|u_0\|_{H^s}+1)^{2k_s}}\|u_0\|_{H^s}\leq \dfrac{2C_s\|u_0\|_{H^s}}{A_s\|u_0\|_{H^s}+1}\leq \dfrac{2C_s}{A_s}\leq \dfrac{2C_s}{8(1+C_0)C_s}<\dfrac14.
\end{align*}
Thus, since $T<1$,
$$\dfrac34 f(T)\leq 2C_s T^{k_s}\|u_0\|_{H^s}\leq 2C_s \|u_0\|_{H^s},$$
and in consequence
$$f(T)\leq \dfrac 83 C_s\|u_0\|_{H^s},$$
which completes the  proof of Lemma \ref{L5.3}.
\end{proof}

\subsection{Uniqueness }

Let $u_1$ and $u_2$ be two solutions of the fractional two-dimensional BO equation in $C([0,T];H^s(\mathbb R^2))\cap L^1((0,T);W^{1,\infty}_{xy}(\mathbb R^2))$ with $s>s_{\alpha}$, corresponding to initial data $\varphi_1$ and $\varphi_2$, respectively. Let
\begin{align}
K:=\max\{\|\nabla u_1\|_{L^1_TL^\infty_{xy}},\|\nabla u_2\|_{L^1_T L^\infty_xy}\},\label{5.4}
\end{align}
and $v:=u_1-u_2$. Then $v$ satisfies the differential equation
$$v_t+D_x^{\alpha} v_x +\mathcal Hv_{yy}+u_1 (u_1)_x-u_2 (u_2)_x=0,$$
i.e.,
\begin{align}
v_t+D_x^{\alpha} v_x +\mathcal Hv_{yy}+\dfrac12\partial_x((u_1+u_2)v)=0,\label{5.5}
\end{align}
and $v(0)=\varphi_1-\varphi_2$. Let us estimate $\|v(t)\|_{L^2_{xy}}$. We multiply each term in \eqref{5.5} by $v(t)$, integrate in $\mathbb R^2_{xy}$, use \eqref{SA} and integration by parts, to obtain
\begin{align*}
\int\dfrac{dv}{dt}v dx dy&=\dfrac12 \dfrac d{dt}\|v(t)\|^2_{L^2_{xy}},\\
\int (D_x^{\alpha} v_x +\mathcal Hv_{yy}) v&=0,\\
\int\dfrac12\partial_x((u_1+u_2)v)v&=-\dfrac12\int(u_1+u_2)v v_x=\dfrac14\int\partial_x(u_1+u_2)v^2 dxdy.
\end{align*}
Hence
$$\dfrac12\dfrac d{dt}\|v(t)\|^2_{L^2_{xy}}=-\dfrac 14\int \partial_x(u_1+u_2)v^2 dxdy \leq\dfrac14\|\partial_x (u_1+u_2)(t)\|_{L^\infty_{xy}}\|v(t)\|^2_{L^2_{xy}},$$
i.e.,
$$\dfrac d{dt}\|v(t)\|^2_{L^2_{xy}}\leq\dfrac12\|\partial_x(u_1+u_2)(t)\|_{L^\infty_{xy}} \|v(t)\|^2_{L^2_{xy}}.$$
Applying Gronwall's inequality we conclude that
\begin{align}
\|v(t)\|^2_{L^2_{xy}}\leq \|\varphi_1-\varphi_2\|^2_{L^2_{xy}} \exp\left(\int_0^t \dfrac12 \|\partial_x(u_1+u_2)(t')\|_{L^\infty_{xy}}dt'\right)\leq\|\varphi_1-\varphi_2\|^2_{L^2_{xy}}e^K.\label{5.6}
\end{align}
Estimative \eqref{5.6} implies the uniqueness in Theorem \ref{T4.1}, choosing $\varphi_1=\varphi_2=u_0$.

\subsection{Existence}

Let $s>s_{\alpha}$ and $u_0\in H^s(\mathbb R^2)$. We will use the Bona-Smith argument. Let us regularize the initial datum as follows. Let $\rho\in C^\infty_0(\mathbb R)$ such that $0\leq\rho\leq1$, $\rho(x)=0$ if $|x|>1$, and $\rho(x)=1$ if $|x|<1/2$. We define $\widetilde \rho(\xi,\eta):=\rho(\sqrt{\xi^2+\eta^2})$ and, for each $n\in\mathbb N$,
$$u_{0,n}(x,y):=\left\{\widetilde\rho\left(\dfrac\xi n,\dfrac \eta n\right)\widehat u_0 (\xi,\eta) \right\}^\vee(x,y).$$
It can be seen that, for each $n\in\mathbb N$, $u_{0,n}\in H^\infty(\mathbb R^2)$ and that  $u_{0,n}\overset{n\to\infty}\longrightarrow u_0$ in $H^s(\mathbb R^2)$. Now, for each $n\in\mathbb N$, we consider the solution $u_n$ of the IVP associated to the fractional two-dimensional BO equation with initial datum $u_{0,n}$, i.e. $u_n$ is the solution of the IVP
\begin{align*}
\left. \begin{array}{rl}
u_t+D_x^{\alpha} u_x +\mathcal Hu_{yy}+u \partial_x u &\hspace{-2mm}=0,\qquad\qquad (x,y)\in\mathbb R^2,\; t\geq 0,\\
u(x,y,0)&\hspace{-2mm}=u_{0,n}(x,y).
\end{array} \right\}
\end{align*}
The existence of the solutions $u_n$ is guaranteed by Lemma \ref{L5.2}.\\

From Lemma \ref{L5.3} we have that $u_n\in C([0,T_n];H^\infty(\mathbb R^2))$, where $T_n:=(A_s\|u_{0,n}\|_{H^s}+1)^{-2}$. Since $u_{0,n}\overset{n\to\infty}\longrightarrow u_0$ in $H^s(\mathbb R^2)$, there exists $N\in\mathbb N$ such that for all $n\geq N$, $T_n\geq (2A_s\|u_0\|_{H^s}+1)^{-2}$. In consequence, for all $n\geq N$, $u_n\in C([0,T];H^\infty(\mathbb R^2))$, where $T:=(2A_s\|u_0\|_{H^s}+1)^{-2}$. Without loss of generality we assume that, for each $n\in\mathbb N$, $u_n\in C([0,T];H^\infty(\mathbb R^2))$. Besides, from Lemma \ref{L5.3}, we can suppose that, for each $n\in\mathbb N$,
\begin{align}
\|u_n\|_{L^\infty_T H^s}\leq 4\|u_0\|_{H^s},\label{5.14}
\end{align}
and
\begin{align}
\|u_n\|_{L^1_TL^\infty_{xy}}+\|\nabla u_n\|_{L^1_T L^\infty_{xy}}\leq \dfrac 83C_s\|u_{0,n}\|_{H^s}\leq\dfrac{16}3 C_s\|u_0\|_{H^s}\equiv \widetilde K.\label{5.15}
\end{align}
From \eqref{5.14} it follows that the sequence $\{u_n\}$ is bounded in $L^\infty([0,T];H^s(\mathbb R^2))$. Therefore, there exist a subsequence of $\{u_n\}$, which we continue denoting by $\{u_n\}$, and a function $u\in L^\infty([0,T];H^s(\mathbb R^2))$ such that $u_n \overset *\rightharpoonup u$ in $L^\infty([0,T];H^s(\mathbb R^2))$, when $n\to\infty$ (weak $\ast$ convergence in $L^\infty([0,T];H^s(\mathbb R^2))$). We will prove that $u\in C([0,T];H^s(\mathbb R^2))\cap L^1([0,T];W^{1,\infty}_{xy}(\mathbb R^2))$ and that $u$ is the solution of the IVP \eqref{BO}.\\

First of all, let us see that $\{u_n\}$ is a Cauchy sequence in $C([0,T];L^2(\mathbb R^2))$. For $m\geq n\geq 1$ we define $v_{n,m}:=u_n-u_m$. Let us prove that
\begin{align}
\|v_{n,m}\|_{L^\infty_T L^2_{xy}}\leq e^{c\widetilde K}\|u_{0,n}-u_{0,m}\|=o(n^{-s}).\label{5.18}
\end{align}
It is clear that $v_{n,m}$ satisfies the differential equation
$$\partial_t v_{n,m}+(D_x^{\alpha} \partial_x +\mathcal H\partial_{yy})v_{n,m}+\dfrac12\partial_x((u_n+u_m)v_{n,m})=0.$$
We multiply this differential equation by $v_{n,m}(t)$ and integrate in $\mathbb R^2_{xy}$ to obtain
$$\dfrac12\dfrac d{dt}\|v_{n,m}(t)\|^2_{L^2}+\dfrac12\int_{\mathbb R^2}\partial_x((u_n+u_m)(t)v_{n,m}(t))v_{n,m}(t)dxdy=0,$$
i.e.,
$$\dfrac12\dfrac d{dt}\|v_{n,m}(t)\|^2_{L^2}+\dfrac14\int_{\mathbb R^2}\partial_x(u_n+u_m)(t)v_{n,m}^2(t)dxdy=0.$$
Hence
$$\dfrac d{dt}\|v_{n,m}(t)\|^2_{L^2}=-\dfrac12 \int_{\mathbb R^2}\partial_x(u_n+u_m)(t)v_{n,m}^2(t)dxdy\leq\dfrac12\|\partial_x(u_n+u_m)(t)\|_{L^\infty_{xy}}\|v_{n,m}(t)\|^2_{L^2}.$$
By Gronwall's inequality we conclude, for $t\in [0,T]$, that
\begin{align*}
\|v_{n,m}(t)\|_{L^2}^2&\leq \|v_{n,m}(0)\|^2_{L^2}\exp \left(\dfrac12\int_0^t \|\partial_x(u_n+u_m)(s)\|_{L^\infty_{xy}} ds \right)\\
&\leq \|u_{0,n}-u_{0,m}\|^2_{L^2}\exp \left(\dfrac12 \|\partial_x u_n\|_{L^1_T L^\infty_{xy}}+\dfrac12 \|\partial_x u_m\|_{L^1_T L^\infty_{xy}} \right)\leq \|u_{0,n}-u_{0,m}\|^2_{L^2}e^{\widetilde K},
\end{align*}
where $\widetilde K$ is the constant in \eqref{5.15}. This way, $\|v_{n,m}\|_{L^\infty_T L^2_{xy}}\leq e^{\widetilde K/2}\|u_{0,n}-u_{0,m}\|_{L^2}$. Estimation \eqref{5.18} follows if we prove that $\|u_{0,n}-u_{0,m}\|_{L^2}=o(n^{-s})$. In fact,
\begin{align*}
\|u_{0,n}-u_{0,m}\|^2_{L^2}&=\qquad\hspace{-.8cm}\underset{\frac n2\leq\sqrt{\xi^2+\eta^2}\leq m}\int \hspace{-.5cm}  \dfrac{\left| \rho\left(\frac{\sqrt{\xi^2+\eta^2}}n \right)-\rho\left(\frac{\sqrt{\xi^2+\eta^2}}m \right) \right|^2}{(1+\xi^2+\eta^2)^s} (1+\xi^2+\eta^2)^s |\widehat u_0(\xi,\eta)|^2 d\xi d\eta\\
&\leq C \qquad\hspace{-.8cm}\underset{\frac n2\leq\sqrt{\xi^2+\eta^2}\leq m}\int \dfrac1{\left(1+\frac{n^2}4\right)^s} (1+\xi^2+\eta^2)^s |\widehat u_0(\xi,\eta)|^2 d\xi d\eta\\
&\leq\dfrac{C}{n^{2s}}\qquad\hspace{-.8cm}\underset{\frac n2\leq\sqrt{\xi^2+\eta^2}\leq m}\int (1+\xi^2+\eta^2)^s |\widehat u_0(\xi,\eta)|^2 d\xi d\eta.
\end{align*}
Since $u_0\in H^s(\mathbb R^2)$, by the Dominated Convergence Theorem, it is clear that 
$$\lim_{n\to\infty}\hspace{-.5cm}\underset{\frac n2\leq\sqrt{\xi^2+\eta^2}\leq m}\int \hspace{-.5cm}(1+\xi^2+\eta^2)^s |\widehat u_0(\xi,\eta)|^2 d\xi d\eta=0.$$
Hence $\|u_{0,n}-u_{0,m}\|_{H^s}=o(n^{-s})$ when $n\to\infty$.\\

For $0\leq \sigma<s$, making interpolation and taking into account inequalities \eqref{5.14} and \eqref{5.18} it follows that
$$\|J^\sigma v_{n,m}\|_{L^\infty_T L^2_{xy}}\leq \|J^s v_{n,m}\|^{\sigma/s}_{L^\infty_T L^2_{xy}} \|v_{n,m}\|^{1-\sigma/s}_{L^\infty_T L^2_{xy}}\leq (8 \|u_0\|_s)^{\sigma/s}\|v_{n,m}\|^{(1-\sigma/s)}_{L^\infty_T L^2_{xy}}.$$

Since $\displaystyle{\lim_{n\to\infty} n^s \|v_{n,m}\|_{L^\infty_T L^2_{xy}}=0}$, we have that $\displaystyle{\lim_{n\to\infty} n^{s(1-\sigma/s)}\|v_{n,m}\|^{(1-\sigma/s)}_{L^\infty_T L^2_{xy}}=0}$, i.e.
$$\|v_{n,m}\|^{(1-\sigma/s)}_{L^\infty_T L^2_{xy}}=o(n^{-(s-\sigma)}).$$

Therefore
\begin{align}
\|J^\sigma v_{n,m}\|_{L^\infty_T L^2_{xy}}\leq C o(n^{-(s-\sigma)}).\label{5.19}
\end{align}

In particular, from \eqref{5.19} it follows that $\{u_n\}$ is a Cauchy sequence in $C([0,T];H^\sigma(\mathbb R))$ for $0\leq\sigma<s$. Let us prove now that $\{u_n\}$ is a Cauchy sequence in $L^1([0,T];W^{1,\infty}_{xy})$, which implies that $u\in L^1([0,T];W_{xy}^{1,\infty})$. We know that
$$v_{n,m}(t)=U_{\alpha}(t)v_{n,m}(0)+\int_0^t U_{\alpha}(t-t')\left(-\dfrac12\partial_x((u_n+u_m)(t')v_{n,m}(t'))\right) dt'.$$

Hence, using Cauchy-Schwarz's inequality  in the variable $t$ and \eqref{lessreg}, it follows that for $\delta>0$
\begin{align*}
\|v_{n,m}\|_{L^1_T L^\infty_{xy}}&\leq \|U_{\alpha}(\cdot_t)v_{n,m}(0)\|_{L^1_T L^\infty_{xy}}+\int_0^T \left\|U_{\alpha}(\cdot_t-t')\left(-\dfrac12\partial_x((u_n+u_m)(t')v_{n,m}(t'))\right)\right\|_{L^1_T L^\infty_{xy}}dt'\\
&\leq T^{1/2}\|U_{\alpha}(\cdot_t)v_{n,m}(0)\|_{L^2_T L^\infty_{xy}}+T^{1/2}\int_0^T\left\|U_{\alpha}(\cdot_t-t')\left(-\dfrac12\partial_x((u_n+u_m)(t')v_{n,m}(t'))\right)\right\|_{L^2_T L^\infty_{xy}}dt'\\
&\lesssim T^{1/2+\tilde k_\delta}\|v_{n,m}(0)\|_{H^{1/4(1-\alpha)+\delta}(\mathbb R^2)}+T^{1/2+\tilde k_\delta} \int_0^T \|\partial_x ((u_n+u_m)(t')v_{n,m}(t'))\|_{H^{1/4(1-\alpha)+\delta}(\mathbb R^2)}dt'\\
&\lesssim T^{1/2+\tilde k_\delta}\|v_{n,m}(0)\|_{H^{1/4(1-\alpha)+\delta}(\mathbb R^2)}+T^{1/2+\tilde k_\delta} \int_0^T \|(u_n+u_m)(t')v_{n,m}(t')\|_{H^{1+1/4(1-\alpha)+\delta}(\mathbb R^2)}dt'.\\
\end{align*}

Taking into account that $H^{1+1/4(1-\alpha)+\delta}(\mathbb R^2)$ is  an algebra, from the last inequality we conclude that

\begin{align*}
\|v_{n,m}\|&_{L^1_T L^\infty_{xy}}\lesssim T^{1/2+\tilde k_\delta}\|v_{n,m}(0)\|_{H^{1/4(1-\alpha)+\delta}(\mathbb R^2)}\\
&+T^{1/2+\tilde k_\delta} \int_0^T \|(u_n+u_m)(t')\|_{H^{1+1/4(1-\alpha)+\delta}(\mathbb R^2)}\|v_{n,m}(t')\|_{H^{1+1/4(1-\alpha)+\delta}(\mathbb R^2)}dt'\\
&\lesssim T^{1/2+\tilde k_\delta}\|v_{n,m}(0)\|_{H^{1/4(1-\alpha)+\delta}(\mathbb R^2)}+T^{3/2+\tilde k_\delta} \|u_n+u_m\|_{L^\infty_T H^{1+1/4(1-\alpha)+\delta}(\mathbb R^2)}\|v_{n,m}\|_{L^\infty_T H^{1+1/4(1-\alpha)+\delta}(\mathbb R^2)}.
\end{align*}

Since, for $0<\delta< s-1-1/4(1-\alpha)$, $\{u_n\}$ is a Cauchy sequence in $L^\infty([0,T];H^{1+1/4(1-\alpha)+\delta}(\mathbb R^2))$, from the last inequality we conclude that $\{u_n\}$ is a Cauchy sequence in $L^1([0,T];L^\infty_{xy})$. On the other hand, from Lemma \ref{L4.11}, it follows that
\begin{align*}
\|\nabla v_{n,m}\|_{L^1_T L^\infty_{xy}}&\leq C_\delta T^{k_\delta}  \left(\|v_{n,m}\|_{L^\infty_T H^{s_{\alpha}+2\delta}} +\int_0^T \|\partial_x ((u_n+u_m)(t)v_{n,m}(t))\|_{H^{s_{\alpha}-1+2\delta}} dt\right)\\
&\lesssim T^{k_\delta}  \left(\|v_{n,m}\|_{L^\infty_T H^{s_{\alpha}+2\delta}} +\int_0^T \| (u_n+u_m)(t)v_{n,m}(t)\|_{H^{s_{\alpha}+2\delta}} dt\right)\\
&\lesssim T^{k_\delta}  \left(\|v_{n,m}\|_{L^\infty_T H^{s_{\alpha}+2\delta}} +T \| u_n+u_m\|_{L^\infty_T H^{s_{\alpha}+2\delta}} \|v_{n,m}\|_{L^\infty_T H^{s_{\alpha}+2\delta}}\right).
\end{align*}

If $\delta>0$ is such that $s_{\alpha}+2\delta<s$, then $\{u_n\}$ is a Cauchy sequence in $L^\infty([0,T];H^{s_{\alpha}+2\delta}(\mathbb R^2))$. Therefore, from the last inequality it follows that $\{\nabla u_n\}$ is a Cauchy sequence in $L^1([0,T];L^\infty_{xy})$. Hence, we have that $\{u_n\}$ is a Cauchy sequence in $L^1([0,T];W_{xy}^{1,\infty}(\mathbb R^2))$. It remains to show that $u\in C([0,T];H^s(\mathbb R^2))$ and that $u$ is a solution of the IVP \eqref{BO}.\\

We first show that $u:[0,T]\to H^s(\mathbb R^2)$ is weakly continuous ($u\in C_w([0,T];H^s(\mathbb R^2))$), i.e. that for each $\phi\in H^s$, the function $t\mapsto(u(t),\phi)_s$ is continuous from $[0,T]$ in $\mathbb R$ (here $(\cdot,\cdot)_s$ denotes the inner product in $H^s$). In order to show this, it is enough to prove that if $f_n:[0,T]\to\mathbb R$ is defined by $f_n(t):=(u_n(t),\phi)_s$, then the sequence $\{f_n\}$ is uniformly Cauchy in $[0,T]$. For $\phi\in H^s$, we take $\psi\in C_0^\infty(\mathbb R^2)$ such that $\|\phi-\psi\|_{H^s}<\eta$ (where $\eta>0$ will be specified later). Hence, for $t\in [0,T]$,
\begin{align}
\notag |f_n(t)-f_m(t)|&=|(u_n(t)-u_m(t),\phi)_s|\leq |(u_n(t)-u_m(t),\phi-\psi)_s|+|(u_n(t)-u_m(t),\psi)_s|\\
\notag&\leq \|u_n(t)-u_m(t)\|_{H^s}\eta+|(u_n(t)-u_m(t),J^{2s}\psi)_0|\\
&\leq C\eta+\|u_n(t)-u_m(t)\|_{L^\infty_{xy}} \|J^{2s} \psi\|_{L^1_{xy}}.\label{ago23_1}
\end{align}

For $\sigma\in(1,s)$, we proved that $\{u_n\}$ es a Cauchy sequence in $C([0,T];H^\sigma(\mathbb R^2))$. Taking into account the immersion $H^\sigma(\mathbb R^2)\hookrightarrow L^\infty(\mathbb R^2)$ we conclude that $u_n(t)\to u(t)$ in $L^\infty_{xy}$ uniformly for $t\in[0,T]$. Given $\epsilon>0$, define $\eta$ such that $C\eta<\dfrac\epsilon2$ and define $N\in\mathbb N$ such that for each $m\geq n\geq N$ and $t\in[0,T]$,
$$\|u_n(t)-u_m(t)\|_{L^\infty_{xy}}\|J^{2s}\psi\|_{L^1_{xy}}<\dfrac\epsilon2.$$

Then, from inequality \eqref{ago23_1} it follows that, for $m\geq n\geq N$ and $t\in[0,T]$, $|f_n(t)-f_m(t)|<\epsilon$. i.e., $\{f_n\}$ is uniformly Cauchy in $[0,T]$, so $u\in C_w([0,T];H^s(\mathbb R^2))$.

Now we show that $u\in C([0,T];H^s(\mathbb R))$. Since $v(x,y,t):=u(x,y,-t)$ satisfies the equation $v_t-D_x^{\alpha} v_x -\mathcal Hu_{yy}-v \partial_x v=0$, it is enough to prove the continuity from the right.  We only prove the continuity from the right at $t_0=0$, being the argument similar for $t_0>0$. Suppose that $t_n\to 0^+$. Then, by \eqref{4.34} and \eqref{5.14} we have that
\begin{align*}
\|u_m(t_n)\|^2_{H^s}&\leq \|u_{0,m}\|^2_{H^s}+C_0 \|\nabla u_m\|_{L^1([0,t_n];L^\infty_{xy})}\|u_m\|^2_{L^\infty([0,t_n];H^s)}\\
&\leq \|u_0\|^2_{H^s}+C_0 \|\nabla u_m\|_{L^1([0,t_n];L^\infty_{xy})}\|u_m\|^2_{L^\infty([0,t_n];H^s)}\\
&\leq \|u_0\|^2_{H^s}+C\|u_0\|^2_{H^s}\|\nabla u_m\|_{L^1([0,t_n];L^\infty_{xy})}.
\end{align*}

Since $u_m(t_n)\rightharpoonup u(t_n)$ when $m\to\infty$ weakly in $H^s$ and $u_m\to u$ in $L^1([0,T];W_{xy}^{1,\infty}(\mathbb R^2))$, then
$$\|u(t_n)\|^2_{H^s}\leq \liminf_{m\to\infty} \|u_m(t_n)\|^2_{H^s}\leq \|u_0\|^2_{H^s}+C\|u_0\|^2_{H^s}\|\nabla u\|_{L^1([0,t_n];L^\infty_{xy})}.$$
Taking into account that $\displaystyle{\lim_{n\to\infty}\|\nabla u\|_{L^1([0,t_n];L^\infty_{xy})}=0}$, it follows that
$$\limsup_{n\to\infty}\|u(t_n)\|^2_{H^s}\leq \|u_0\|^2_{H^s},$$
which is equivalent to
$$\limsup_{n\to\infty}\|u(t_n)\|_{H^s}\leq \|u_0\|_{H^s}.$$

Since $u_{0,m}\rightharpoonup u_0$ when $m\to\infty$ in $H^s$ and $u_{0,m}=u_m(0)\rightharpoonup u(0)$ when $m\to\infty$ in $H^s$, then $u(0)=u_0$. And, since $u(t_n)\rightharpoonup u(0)$ when $n\to\infty$ in $H^s$,
$$\limsup_{n\to\infty}\|u(t_n)\|_{H^s}\leq \|u(0)\|_{H^s}\leq \liminf_{n\to\infty} \|u(t_n)\|_{H^s},$$
which implies that $\displaystyle{\lim_{n\to \infty}\|u(t_n)\|_{H^s}=\|u(0)\|_{H^s}}$. This, together with the fact that $u(t_n)\rightharpoonup u(0)$ in $H^s$, imply that $u(t_n)\to u_0$ in $H^s$, i.e., $u:[0,T]\to H^s$ is continuous from the right at $t_0=0$.\\

On the other hand, it is clear that $\partial_t u_n+(D_x^{\alpha} \partial_x +\mathcal H\partial_{yy}) u_n$ converges to $\partial_t u+(D_x^{\alpha} \partial_x +\mathcal H\partial_{yy}) u$ in the sense of distributions. Besides, $u_n\partial_x u_n\to u\partial_x u$ in $C([0,T];L^2_{xy})$ because $u_n\to u$ in $C([0,T];L^{\infty}_{xy})$ and $\partial_xu_n\to\partial_xu$ in $C([0,T];L^2_{xy})$. In consequence, since $\partial_t u_n+(D_x^{\alpha} \partial_x +\mathcal H\partial_{yy}) u_n+u_n \partial_x u_n=0$, we have that $u$ satisfies the fractional two-dimensional BO equation in the sense of distributions.

\subsection{Continuous dependence on the initial data}

In this subsection we follow the same procedure of Koch and Tzvetkov in \cite{KT2003}. For the sake of completeness we include the proofs of Lemmas \ref{KTL3.6} and \ref{KTL4.1}. Let us recall the notations introduced at the beginning of Lemma \ref{L4.11}: for $\lambda:=2^k$, with $k\in\mathbb N$, we define $u_\lambda:=\Delta_\lambda u$, where $\displaystyle{(\Delta_\lambda u)^\wedge(\xi,\eta):=\varphi \left(\dfrac\xi{2^k},\dfrac\eta{2^k} \right)\widehat u(\xi,\eta)}$. Besides we define $u_1:=\Delta_1 u$, where $(\Delta_1 u)^\wedge (\xi,\eta)=\chi(\xi,\eta) \widehat u(\xi,\eta)$ and $\displaystyle{\chi(\xi,\eta)+\sum_{k=1}^\infty \varphi\left(\dfrac\xi{2^k},\dfrac\eta{2^k} \right)}=1$. We also define $\tilde \Delta_\lambda:=\Delta_{\lambda/2}+\Delta_\lambda+\Delta_{2\lambda}$, for $\lambda$ dyadic greater than 1, and $\tilde \Delta_1:=\Delta_1+\Delta_2$.

\begin{lemma}\label{CP}(see \cite{CP2016}) There is a constant $C>0$ such that for any $w\in L^2(\mathbb R^2)$ and any $v$ such that $\nabla v\in L^{\infty}(\mathbb R^2)$
\[\|[\Delta_{\lambda}, v\partial_x]w\|_{L^2}\leq C\|\nabla v\|_{L^{\infty}}\|w\|_{L^2}\,,\]
where $[\Delta_{\lambda}, v\partial_x]w:=\Delta_{\lambda}(v\partial_xw)-v\partial_x\Delta_{\lambda}w$.
\end{lemma} 

\begin{lemma} \label{KTL3.6} (see \cite{KT2003}) Let $u\in C([0,T];H^s(\mathbb R^2))\cap L^1([0,T];W^{1,\infty}_{xy}(\mathbb R^2))$ be a solution of the  fractional two-dimensional BO equation in $[0,T]$ with $s>s_{\alpha}$. Let $\delta$ and $k$ be such that $1<\delta<k$. Assume that the dyadic sequence of positive numbers $(w_\lambda)_\lambda$ satisfies $\delta w_\lambda\leq w_{2\lambda}\leq k w_\lambda$, for each dyadic integer $\lambda$ and that $\sum_{\lambda} w_\lambda^2 \|u_\lambda(0)\|^2_{L^2}<\infty$. Then for each $\tau$ and $t$ in $[0,T]$ it follows that
\begin{align}
\sum_{\lambda} w_\lambda^2 \|u_\lambda(t)\|^2_{L^2} \lesssim \exp \left(C\|\nabla u\|_{L^1_I L^\infty_{xy}} \right) \sum_{\lambda} w_\lambda^2 \|u_\lambda(\tau)\|^2_{L^2}, \label{KT3.33}
\end{align}
where $I:=[\tau,t]$.
\end{lemma}

\begin{proof} It is easy to see that $u_\lambda$ satisfies the differential equation
$$\partial_t u_\lambda +(D_x^{\alpha} \partial_x +\mathcal H\partial_{yy}) u_\lambda=-u \partial_x u_\lambda-[\Delta_\lambda,u \partial_x]u.$$
We multiply this equation by $u_\lambda(t)$, integrate in $\mathbb R^2_{xy}$ and use integration by parts to conclude that
$$\dfrac12 \dfrac d{dt} \|u_\lambda(t)\|^2_{L^2}=\dfrac 12\int_{\mathbb R^2} u_x(t)(u_\lambda(t))^2+\int_{\mathbb R^2}-[\Delta_\lambda,u(t)\partial_x]u(t)u_\lambda(t).$$
Let us integrate this equation with respect to time in $[\tau,t]$ to obtain
$$\dfrac12\|u_\lambda(t)\|_{L^2}^2 =\dfrac12 \|u_\lambda(\tau)\|^2_{L^2}+\dfrac12\int_\tau^t\int_{\mathbb R^2}u_x(\sigma)(u_\lambda(\sigma))^2dxdyd\sigma-\int_\tau^t\int_{\mathbb R^2}[\Delta_\lambda,u(\sigma)\partial_x]u(\sigma)u_\lambda(\sigma)dxdyd\sigma.$$
Hence,
\begin{align}
\notag \sum_\lambda w_\lambda^2 \|u_\lambda(t)\|^2_{L^2}\leq &\sum_\lambda w_\lambda^2 \|u_\lambda(\tau)\|^2_{L^2}+\sum_\lambda w_\lambda^2 \int_\tau^t \|\nabla u(\sigma)\|_{L^\infty} \|u_\lambda(\sigma)\|^2_{L^2} d\sigma\\
&+2\sum_\lambda w_\lambda^2 \int_\tau^t \|u_\lambda(\sigma)\|_{L^2}\|[\Delta_\lambda,u(\sigma)\partial_x]u(\sigma)\|_{L^2}d\sigma.\label{sep06_1}
\end{align}
Taking into account that $\Delta_\lambda \tilde \Delta_\lambda=\Delta_\lambda$, it could be seen that
$$[\Delta_\lambda,u\partial_x]u=[\Delta_\lambda,u\partial_x ]\tilde \Delta_\lambda u+\Delta_\lambda(u\partial_x (I-\tilde \Delta_\lambda)u).$$
Therefore,
\begin{align}
\|[\Delta_\lambda,u(\sigma)\partial_x]u(\sigma)\|_{L^2}\leq \|[\Delta_\lambda,u(\sigma)\partial_x]\tilde \Delta_\lambda u(\sigma)\|_{L^2}+\|\Delta_\lambda(u(\sigma)\partial_x(I-\tilde\Delta_\lambda)u(\sigma))\|_{L^2}.\label{sep06_2}
\end{align}
Now, from Lemma \ref{CP}, it follows that
\begin{align}
\|[\Delta_\lambda,u(\sigma)\partial_x]\tilde\Delta_\lambda u(\sigma)\|_{L^2}\leq \|\nabla u(\sigma)\|_{L^\infty}\|\tilde\Delta_\lambda u(\sigma)\|_{L^2}.\label{sep06_3}
\end{align}
On the other hand, since for $\mu$ dyadic, $\mu\leq\dfrac\lambda{16}$, we have that $\Delta_\lambda(u_\mu \partial_x(I-\tilde \Delta_\lambda)u)=0$ and the operator $(I-\tilde\Delta_\lambda)$ is bounded from $L^\infty$ to $L^\infty$, we see that
\begin{align}
\notag \|\Delta_\lambda(u\partial_x(I-\tilde\Delta_\lambda)u)\|_{L^2}&\leq \sum_{\mu\geq \lambda/8} \|\Delta_\lambda(u_\mu \partial_x (I-\tilde \Delta_\lambda)u)\|_{L^2}  \leq \sum_{\mu\geq \lambda/8} \|u_\mu \partial_x(I-\tilde \Delta_\lambda)u\|_{L^2}\\
\notag &\leq \sum_{\mu\geq \lambda/8} \|(I-\tilde\Delta_\lambda)\partial_x u\|_{L^\infty}\|u_\mu\|_{L^2}\lesssim  \sum_{\mu\geq \lambda/8} \|\partial_x u\|_{L^\infty}\|u_\mu\|_{L^2}\\
&\leq \|\nabla u\|_{L^\infty} \sum_{\mu\geq \lambda/8}\|u_\mu\|_{L^2}.\label{sep06_4}
\end{align}

From \eqref{sep06_2}, \eqref{sep06_3}, and  \eqref{sep06_4}, it follows that
\begin{align}
\notag\int_\tau^t \sum_\lambda & w_\lambda^2 \|u_\lambda(\sigma)\|_{L^2 }\|[\Delta_\lambda,u(\sigma)\partial_x]u(\sigma)\|_{L^2}d\sigma\\
\notag\leq &\int_\tau^t \sum_\lambda w_\lambda^2 \|u_\lambda(\sigma)\|_{L^2}\left( \|\nabla u(\sigma)\|_{L^\infty} \|\tilde \Delta_\lambda u(\sigma)\|_{L^2} +\|\nabla u(\sigma)\|_{L^\infty} \sum_{\mu\geq \lambda/8} \|u_\mu(\sigma)\|_{L^2}  \right)d\sigma\\
\notag\leq &\int_\tau^t \|\nabla u(\sigma)\|_{L^\infty} \left(\sum_\lambda w_\lambda^2 \|u_\lambda(\sigma)\|_{L^2} \|\tilde\Delta_\lambda u(\sigma)\|_{L^2} +\sum_\lambda \left( w_\lambda^2 \|u_\lambda(\sigma)\|_{L^2} \sum_{\mu\geq \lambda/8} \|u_\mu(\sigma)\|_{L^2}\right) \right) d\sigma.\\
\leq&2\int_\tau^t \|\nabla u(\sigma)\|_{L^\infty} \sum_\lambda \left( w_\lambda^2 \|u_\lambda(\sigma)\|_{L^2} \sum_{\mu\geq\lambda/8} \|u_\mu(\sigma)\|_{L^2} \right) d\sigma.\label{sep06_5}
\end{align}
From \eqref{sep06_1} and \eqref{sep06_5} we conclude that
\begin{align}
\sum_\lambda w_\lambda^2 \|u_\lambda(t)\|_{L^2}^2 \leq \sum_\lambda w_\lambda^2 \|u_\lambda(\tau)\|^2_{L^2}+C\int_\tau^t \|\nabla u(\sigma)\|_{L^\infty}\sum_\lambda \left(w_\lambda^2 \|u_\lambda(\sigma)\|_{L^2}\sum_{\mu\geq \lambda/8} \|u_\mu(\sigma)\|_{L^2}\right)d\sigma.\label{sep26}
\end{align}

Since $\delta w_\lambda\leq w_{2\lambda}$ for each $\lambda$, it can be proved by induction, that for all $j\geq 0$,
\begin{align}
w_\lambda\leq\dfrac1{\delta^j}w_{2^j\lambda},\label{sep06_6}
\end{align}

and since $w_\lambda\leq k w_{\lambda/2}$, it can also be seen that
\begin{align}
w_\lambda\leq k w_{2^{-1}\lambda},\quad w_\lambda\leq k^2 w_{2^{-2}\lambda},\quad w_\lambda\leq k^3 w_{2^{-3}\lambda}. \label{sep06_7}
\end{align}

Let us see that for any $(d_\lambda)_\lambda$, such that $\displaystyle{\left(\sum_\lambda d_\lambda^2\right)^{1/2}}<\infty$ we have that
\begin{align}
\sum_\lambda \left( \sum_{\mu\geq\lambda/8} w_\lambda \|u_\mu(\sigma)\|_{L^2} \right)d_\lambda\lesssim \left( \sum_\lambda w_\lambda^2 \|u_\lambda(\sigma)\|^2_{L^2}\right)^{1/2}\left(\sum_\lambda d_\lambda^2\right)^{1/2}.\label{sep06_8}
\end{align}

In fact, from \eqref{sep06_6} and \eqref{sep06_7} it follows that
\begin{align*}
\sum_\lambda \left( \sum_{\mu\geq\lambda/8} w_\lambda \|u_\mu(\sigma)\|_{L^2} \right)d_\lambda&=\sum_{\lambda} \left(\sum_{j=-3}^\infty w_\lambda \|w_{2^j \lambda}(\sigma)\|_{L^2}\right)d_\lambda \\
&\lesssim \sum_\lambda \left(\sum_{j=-3}^{-1} k^{-j} w_{2^j\lambda} \|u_{2^j \lambda}(\sigma)\|_{L^2}+\sum_{j=0}^\infty \dfrac1{\delta^j}w_{2^j \lambda} \|u_{2^j \lambda}(\sigma)\|_{L^2}\right)d_\lambda .
\end{align*}

Defining  $a_j:=k^{-j}$ for $j=-3,-2,-1$ and $a_j:=\dfrac1{\delta^j}$ for $j\geq 0$, we have
\begin{align*}
\sum_\lambda \left( \sum_{\mu\geq\lambda/8} w_\lambda \|u_\mu(\sigma)\|_{L^2} \right)d_\lambda&\lesssim \sum_\lambda \left(\sum_{j=-3}^\infty a_j w_{2^j\lambda} \|u_{2^j\lambda}(\sigma)\|_{L^2}\right)d_\lambda= \sum_{j=-3}^\infty a_j \left(\sum_\lambda w_{2^j\lambda} \|u_{2^j\lambda}(\sigma)\|_{L^2}d_\lambda\right)\\
&\leq \sum_{j=-3}^\infty a_j \left(\sum_{\lambda}w_{2^j\lambda}^2 \|u_{2^j\lambda}(\sigma)\|^2_{L^2} \right)^{1/2} \left(\sum_\lambda d_\lambda^2 \right)^{1/2}\\
&\leq \left( \sum_{j=-3}^\infty a_j\right)\left(\sum_\lambda w_\lambda^2 \|u_\lambda(\sigma)\|^2_{L^2} \right)^{1/2}\left(\sum_{\lambda}d_\lambda^2 \right)^{1/2}\\
&\lesssim \left(\sum_\lambda w_\lambda^2 \|u_\lambda (\sigma)\|^2_{L^2} \right)^{1/2}\left(\sum_{\lambda}d_\lambda^2\right)^{1/2},
\end{align*}
which proves \eqref{sep06_8}.\\

Taking in \eqref{sep06_8} $d_\lambda:=w_\lambda \|u_\lambda (\sigma)\|_{L^2}$, we conclude that
\begin{align}
\sum_\lambda w_\lambda^2\|u_\lambda(\sigma)\|_{L^2} \left(\sum_{\mu\geq \lambda/8} \|u_\mu(\sigma)\|_{L^2} \right)\lesssim \sum_\lambda w_\lambda^2 \|u_\lambda (\sigma)\|_{L^2}^2.\label{sep13_09}
\end{align}

From \eqref{sep26} and \eqref{sep13_09} it follows that
\begin{align}
\sum_\lambda w_\lambda^2 \|u_\lambda (t)\|_{L^2}^2 \lesssim \sum_\lambda w_\lambda^2 \|u_\lambda(\tau)\|_{L^2}^2+C\int_\tau^t \|\nabla u (\sigma)\|_{L^\infty}\sum_\lambda w_\lambda^2 \|u_\lambda(\sigma)\|_{L^2}^2 d\sigma.\label{3.37}
\end{align}

Hence, \eqref{KT3.33} follows from \eqref{3.37} using Gronwall's inequality.
\end{proof}

\begin{lemma} \label{KTL4.1} (see \cite{KT2003}) Let us assume that $\nu^n\to \nu$ in $H^s(\mathbb R^2)$. Then there exists a dyadic sequence of positive numbers $(w_{\lambda})_{\lambda}$ such that $2^s w_\lambda\leq w_{2\lambda}\leq 2^{s+1}w_\lambda$, and
\begin{align}
\lim_{\lambda\to\infty}\dfrac{w_\lambda}{\lambda^s}=\infty,\label{KT4.19}
\end{align}
such that
\begin{align}
\sup_n \sum_\lambda w_\lambda^2 \|\nu_\lambda^n\|^2_{L^2}<\infty,\label{KT4.20}
\end{align}
where $\nu_\lambda^n:=\Delta_\lambda \nu^n$.
\end{lemma}

\begin{proof} Let us consider the space $l^1(\mathbb N)$ of summable sequences. We define
\begin{align}
a_i^n:=\lambda^{2s}\|\nu^n_{2^i}\|_{L^2}^2\quad \text{and}\quad a_i:=\lambda^{2s}\|\nu_{2^i}\|^2_{L^2},\label{KT4.21}
\end{align}
where $\lambda:=2^i$. From the hypothesis it follows that $\left(a_i^n\right)_{i\in \mathbb N}\to \left(a_i\right)_{i\in \mathbb N}$ in $l^1(\mathbb N)$, as $n\to\infty$. In fact,
\begin{align}
\sum_i |a_i^n-a_i|=\sum_i (2^i)^{2s} \left| \|\nu^n_{2^i}\|^2_{L^2}-\|\nu_{2^i}\|_{L^2}^2 \right|. \label{sep25_01}
\end{align}
On the other hand, from the Littlewood-Paley Theorem, we have that
\begin{align}
\|J^s(\nu^n-\nu)\|_{L^2}\lesssim \left(\sum_i (2^i)^{2s}\|\nu^n_{2^i}-\nu_{2^i}\|_{L^2}^2 \right)^{1/2} \lesssim \|J^s(\nu^n-\nu)\|_{L^2}.\label{sep25_02}
\end{align}
Since $\nu^n\to \nu$ in $H^s(\mathbb R^2)$, $\|J^s(\nu^n-\nu)\|_{L^2}\to 0$ as $n\to\infty$. Therefore, from \eqref{sep25_01} and \eqref{sep25_02}, it follows that $(a_i^n)_{i\in\mathbb N}\to(a_i)_{i\in\mathbb N}$ in $l^1(\mathbb N)$ as $n\to\infty$.\\

Let us construct a sequence $\{\mu_i\}$ such that $0<\mu_i<\mu_{i+1}\leq 2\mu_i$, such that $\displaystyle{\lim_{i\to\infty}\mu_i=\infty}$ and
\begin{align}
\sup_n \sum_{i=1}^\infty \mu_i a_i^n<\infty.\label{KT4.23}
\end{align}
(Observe that $\displaystyle{\sup_n\sum_{i=1}^\infty a_i^n<\infty}$ because $(a_i^n)_{i\in\mathbb N}\to (a_i)_{i\in\mathbb N}$ in $l^1(\mathbb N)$).\\

Let us see that, for each $k\in\mathbb N$, there exists $N_k$ such that
\begin{align}
\sup_n \sum_{i=N_k}^\infty a_i^n<2^{-k}\label{KT4.24}.
\end{align}
In fact, since $\displaystyle{\sum_{i=1}^\infty a_i<\infty}$, there exists $N_k^*$ such that $\displaystyle{\sum_{i=N_k^*}^\infty a_i<\dfrac{2^{-k}}{2}}$. Since $(a_i^n)_{i\in\mathbb N}\to(a_i)_{i\in\mathbb N}$ in $l^1(\mathbb N)$ as $n\to\infty$, there exists $N$ such that, for each $n\geq N$, $\displaystyle{\sum_{i=1}^\infty}|a_i^n-a_i|<\dfrac{2^{-k}}{3}$.\\

On the other hand, there exists $N_k^{**}$ such that, for each $n\in\{1,\dots,N-1\}$, $\displaystyle{\sum_{i=N_k^{**}}a_i^n<\dfrac{2^{-k}}{2}}$. We define $N_k:=\max\{N_k^*,N_k^{**}\}$. Hence, for each $n\geq N$,
$$\sum_{i=N_k}^\infty a_i^n\leq \sum_{i=N_k}^\infty a_i+\sum_{i=N_k}^\infty (a_i^n-a_i)<\dfrac{2^{-k}}2+\dfrac{2^{-k}}3,$$
and for $n\leq N$,
$$\sum_{i=N_k}^\infty a_i^n \leq \sum_{i=N_k^{**}} a_i^n <\dfrac{2^{-k}}2.$$
This way, inequality \eqref{KT4.24} follows.\\

We can assume that the sequence $\{N_k\}_{k\in\mathbb N}$ is strictly increasing. For fixed $i\in\mathbb N$, there exists a unique $k\in\mathbb N$ such that $N_{k-1}\leq i<N_k$. Let us define $\mu_i:=2^{k/2}$. It is clear that $0<\mu_i\leq \mu_{i+1}\leq 2\mu_i$. Using \eqref{KT4.24} we obtain that
\begin{align}
\sum_{i=1}^\infty \mu_i a_i^n=\sum_{k=0}^\infty \sum_{i=N_k}^{N_{k+1}-1} \mu_i a_i^n \leq \sum_{k=0}^\infty 2^{(k+1)/2}\sum_{i=N_k}^{N_{k+1}-1}a_i^n\leq \sum_{k=0}^\infty 2^{(k+1)/2}2^{-k}=2^{1/2}\sum_{k=0}^\infty 2^{-k/2}<\infty. \label{KT4.25}
\end{align}
i.e. $\displaystyle{\sup_n}\sum_i \mu_i \lambda^{2s}\|\nu^n_{2^i}\|^2_{L^2}<\infty$, where $\lambda=2^i$. Define $w_\lambda:= \mu_i^{1/2}\lambda^s$. Since $w_{2\lambda}=\mu_{i+1}^{1/2}(2\lambda)^s=2^s\mu_{i+1}^{1/2}\lambda^s$, then $2^sw_\lambda=2^s\mu_i^{1/2}\lambda^s\leq 2^s\mu_{i+1}^{1/2}\lambda^s=w_{2\lambda}$ and $w_{2\lambda}=2^s \mu_{i+1}^{1/2}\lambda^s\leq 2^s(2\mu_i)^{1/2}\lambda^s=2^{1/2}2^s w_\lambda\leq 2^{s+1}w_\lambda$. Hence
$$\lim_{i\to\infty}\dfrac{w_\lambda}{\lambda^s}=\lim_{i\to\infty}\mu_i^{1/2}=\infty.$$
\end{proof}

\textbf{Proof of the continuous dependence on the initial data}. This proof follows the same ideas of Koch and Tzvetkov in \cite{KT2003}.\\

Since $T=T(\|u_0\|_{H^s})$, given $T'<T$, there exists a neighborhood $\mathcal U$ of $u_0$ in $H^s(\mathbb R^2)$ such that if $\tilde u_0\in\mathcal U$, then the IVP associated to the fractional two-dimensional BO equation with initial datum $\tilde u_0$ has a solution $\tilde u$ in $C([0,T'];H^s(\mathbb R^2))\cap L^1([0,T'];W_{x,y}^{1,\infty}(\mathbb R^2))$. Let us consider a sequence $\{u_{n,0}\}_{n\in\mathbb N}$ in $\mathcal U$, such that $u_{n,0}\to u_0$ in $H^s(\mathbb R^2)$ as $n\to\infty$. Let $u_n\in C([0,T'];H^s(\mathbb R^2))\cap L^1([0,T'];W_{x,y}^{1,\infty}(\mathbb R^2))$ be the solution of the fractional two-dimensional BO equation such that $u_n(0)=u_{n,0}$. Carrying out a procedure as the one used to obtain \eqref{5.18}, it can be proved that
\begin{align}
u_n\to u\quad\text{in $C([0,T'];L^2(\mathbb R^2))$}. \label{KT4.26}
\end{align}

We apply Lemma \ref{KTL4.1} with $\nu^n:=u_{n,0}$ and $\nu:=u_0$ to conclude that there exists a dyadic sequence $(w_\lambda)$ with $2^s w_\lambda\leq w_{2\lambda}\leq 2^{s+1} w_\lambda$ and $\dfrac{w_\lambda}{\lambda^s}\to\infty$ as $\lambda\to\infty$, such that $\displaystyle{\sup_n \sum_\lambda w_\lambda^2 \|(u_{n,0})_\lambda\|^2_{L^2}}<\infty$. By Lemma \ref{KTL3.6}, for $t\in[0,T']$ and $n\in\mathbb N$, we have that
\begin{align*}
\sum_\lambda w_\lambda^2 \|(u_n(t))_\lambda\|^2_{L^2}&\lesssim \exp(C\|\nabla u_n\|_{L^1_{T'}L_{x,y}^\infty}) \sum_\lambda w_\lambda^2 \|(u_{n,0})_\lambda \|_{L^2}^2\lesssim \sum_\lambda w_\lambda^2 \|(u_{n,0})_\lambda \|_{L^2}^2\\
&\lesssim \sup_n\left(\sum_\lambda w_\lambda^2 \|(u_{n,0})_\lambda \|_{L^2}^2 \right)<\infty,
\end{align*} 
where we use Lemma \ref{L5.3} to assure that there exists a constant $C$, independent from $n$, such that $\|\nabla u_n\|_{L^1_{T'}L_{x,y}^\infty}\leq C$.\\

Besides, by Lemma \ref{KTL3.6},
$$\sum_\lambda w_\lambda^2 \|(u(t))_\lambda \|_{L^2}^2\lesssim \sum_\lambda w_\lambda^2 \|(u_{0})_\lambda \|_{L^2}^2.$$
Hence
\begin{align}
\sup_n \sup_{0\leq t\leq T'} \sum_\lambda w_\lambda^2 \left( \|(u_n(t))_\lambda \|_{L^2}^2+\|(u(t))_\lambda\|_{L^2}^2\right)\equiv A^2<\infty. \label{KT4.27}
\end{align}

Let us define
\begin{align}
V_{\Lambda}:=\sum_{\lambda\leq \Lambda}v_\lambda \label{KT4.28}.
\end{align}

By the triangular inequality we obtain
\begin{align}
\|u_n-u\|_{L^\infty_{T'}H^s}\leq \|(U_n)_\Lambda-U_\Lambda\|_{L^\infty_{T'}H^s}+\|(U_n)_\Lambda-u_n\|_{L^\infty_{T'}H^s}+\|U_\Lambda-u\|_{L^\infty_{T'}H^s}.\label{KT4.29}
\end{align}

Let $\epsilon>0$ fixed. We initially prove that there exists a dyadic integer $\Lambda$ such that, for each $t\in[0,T']$,
\begin{align}
\sup_n \|(U_n)_\Lambda(t)-u_n(t)\|_{H^s}+\|U_\Lambda(t)-u(t)\|_{H^s}\lesssim\dfrac\epsilon2.\label{KT4.30}
\end{align}

In fact, using the Littlewood-Paley Theorem, we can conclude that
\begin{align*}
\|(U_n)_\Lambda(t)-u_n(t)\|_{H^s}+\|U_\Lambda(t)-u(t)\|_{H^s}&=\left\|\sum_{\lambda>\Lambda}(u_n(t))_\lambda\right\|_{H^s}+\left\|\sum_{\lambda>\Lambda} (u(t))_\lambda\right\|_{H^s}\\
&\lesssim \left(\sum_{\lambda>\Lambda} w_\lambda^2 \dfrac{\lambda^{2s}}{w_\lambda^2}\|(u_n(t))_\lambda \|_{L^2}^2 \right)^{1/2}+\left(\sum_{\lambda>\Lambda} w_\lambda^2 \dfrac{\lambda^{2s}}{w_\lambda^2}\|(u(t))_\lambda \|_{L^2}^2 \right)^{1/2}.
\end{align*}

Since $\displaystyle{\lim_{\lambda\to\infty}}\dfrac{\lambda^{2s}}{w_\lambda^2}=0$, there exists $\Lambda$ such that if $\lambda>\Lambda$, $\dfrac{\lambda^{2s}}{w_\lambda^2}<\dfrac{\epsilon^2}{16A^2}$. Therefore, if $\lambda>\Lambda$,
\begin{align*}
\|(U_n)_\Lambda(t)-u_n(t)\|_{H^s}+\|U_\Lambda(t)-u(t)\|_{H^s}&\lesssim \dfrac\epsilon{4A}\left(\sum_{\lambda>\Lambda} w_\lambda^2 \|(u_n(t))_\lambda\|_{L^2}^2 \right)^{1/2}+\dfrac\epsilon{4A} \left(\sum_{\lambda>\Lambda} w_\lambda^2 \|(u(t))_\lambda\|_{L^2}^2 \right)^{1/2}\\
&\lesssim \dfrac\epsilon{4A}A+\dfrac\epsilon{4A}A=\dfrac\epsilon2,
\end{align*}
which proves \eqref{KT4.30}. Taking into account \eqref{KT4.26}, there exists $N\in\mathbb N$, such that for each $n\geq N$ and $t\in[0,T']$,
\begin{align}
\|(U_n(t))_\Lambda-(U(t))_\Lambda\|_{H^s}\leq (2\Lambda)^s \|(U_n(t))_\Lambda-(U(t))_\Lambda\|_{L^2}<\dfrac\epsilon2.\label{KT4.31}
\end{align}

From \eqref{KT4.29}, \eqref{KT4.30} and \eqref{KT4.31} it follows that, for each $n\geq N$,
$$\|u_n-u\|_{L^\infty_{T'} H^s}\leq \dfrac \epsilon2+C\dfrac\epsilon2,$$
i.e. $\displaystyle{\lim_{n\to\infty}\| u_n-u\|_{L^\infty_{T'}H^s}}=0$. \qed

\newpage

\textbf{Acknowledgments}\\

Partially supported by Colciencias through grant number FP44842-087-2015 of the Fondo National de Financiamiento Para la Ciencia, Tecnología y la Innovación Francisco José de Caldas (Project Ecuaciones Diferenciales Dispersivas y Elípticas no Lineales).

\end{document}